\documentclass[12pt]{article}

\usepackage{graphicx}
\usepackage{enumerate}
\usepackage{natbib}
\usepackage{url} 

\usepackage{graphicx,bm,url,microtype}
\usepackage{hyperref}
\usepackage{paralist}
\usepackage{diagbox}

\usepackage{authblk}
\usepackage{amsfonts,amssymb,amsthm,amsmath, amsthm}
\usepackage{natbib}
\usepackage[colorinlistoftodos]{todonotes}
\usepackage[utf8]{inputenc}
\newtheorem{proposition}{{\sc\bf Proposition}}
\newtheorem{theorem}{{\sc\bf Theorem}}
\newtheorem{lemma}{{\sc\bf Lemma}}
\newtheorem{remark}{{\sc\bf Remark}}

\begin{document}
\begin{center}
	\Large \bf On a general definition of the functional linear model
\end{center}
 \normalsize
 \begin{center}
 	Jos\'e R. Berrendero$^1$\hspace{.2cm} Alejandro Cholaquidis$^2$, and \hspace{.2cm} Antonio Cuevas$^1$ \\
 	$^1$ Departamento de Matemáticas, Universidad Autónoma de Madrid\\
 	$^2$ Universidad de la República, Uruguay
 \end{center}

\begin{abstract}
	A general formulation of the linear model with functional (random) explanatory variable $X=X(t),\ t\in T$, and scalar response $Y$ is proposed. It includes the standard functional linear model, based on the inner product in the space $L^2[0,1]$, as a particular case. It also includes all models in which $Y$ is assumed to be (up to an additive noise) a linear combination of a finite or countable collections of marginal variables $X(t_j)$, with $t_j\in T$ or a linear combination of a finite number of linear projections of $X$.
	This general formulation can be interpreted in terms of the RKHS space generated by the covariance function of the process $X(t)$. Some consistency results are proved. A few experimental results are given in order to show the practical interest of considering, in a unified framework, linear models based on a finite number of marginals $X(t_j)$ of the process $X(t)$.
\end{abstract}

\noindent%
\vfill
 
\section{Introduction}

Linear regression is a topic of leading interest in statistics. The general paradigm is well-known: one aims to predict a response variable $Y$ in the best possible way as a linear (or affine) function of some explanatory variable $X$. 
In the classical case of multivariate regression, where $Y$ is a real random variable and $X$ takes values in ${\mathbb R}^d$ there is little doubt about the meaning of ``linear''. However, this is not that obvious when $X$ is a more complex object that can be modelled via different mathematical structures. The most important example arises perhaps in the field of Functional Data Analysis (FDA) where $X=X(t)$ is a real function;  see e.g., \cite{cuevas2014partial} for a general survey of FDA and \cite{horvath2012} for a more detailed account, including an up-to-date overview of functional linear models.  

\noindent \textit{Some notation}

More precisely, we will deal here with the scalar-on-function regression problem where the response $Y$ is a real random variable and $X$ is a random function (i.e., a trajectory of a stochastic process). In formal terms,
let $(\Omega,\mathcal{F},\mathbb{P})$ be a probability space and denote by $L^2(\Omega) = L^2(\Omega,\mathcal{F},\mathbb{P})$ the space of square integrable random variables defined on $(\Omega,\mathcal{F},\mathbb{P})$. Denote by $\langle X, Y\rangle =\mathbb{E}(XY)$ the inner product in this space and by $\|\cdot\|$ the corresponding induced norm.

Consider a response variable $Y\in L^2(\Omega)$ and a family of regressors  $\{X(t):\, t\in T\}\subset L^2(\Omega)$, where $T$ is an arbitrary index set. For the sake of simplicity we will assume both the response and the explanatory variable are centered, so that $\mathbb{E}(Y)=\mathbb{E}(X(t))=0$, for all $t\in T$.

The covariance function $K:T\times T\to\mathbb{R}$ of $\{X(t):\, t\in T\}$ is defined by $K(s,t)=\langle X(s), X(t)\rangle =\mathbb{E}(X(s)X(t))$. It can be shown that $K$ is the covariance function of a family of variables if and only if $K$ is symmetric and positive semidefinite.  

\noindent \textit{The aim of this work. Some motivation}

Our purpose is to show that the term ``linear'' admits several interpretations in the functional case; all of them could be useful, depending on the considered context. We will provide a general formulation of the linear model and we will show that several useful formulations of functional linear models appear as particular cases. Some basic consistency results will be given regarding the estimation of the slope (possibly functional) parameter. The theory of Reproducing Kernel Hilbert Spaces (RKHS) will be an important auxiliary tool in our approach; see \cite{berlinet2004}. 

In order to give some perspective and motivation, let us consider, for $T=[0,1]$, the classical $L^2$-based functional regression model, as given by
\begin{equation}\label{eq:l2model}
	Y=\int_0^1\beta(t)X(t)dt+\varepsilon:=\langle \beta,X\rangle_2+\varepsilon,
\end{equation}
where $X=X(t)$ is a process with trajectories in the space $L^2[0,1]$ of square integrable real functions, $\langle \cdot,\cdot\rangle_2$ stands for the inner product in $L^2[0,1]$, $\varepsilon$ is the error variable, independent from $X=X(t)$, with ${\mathbb E}(\varepsilon)=0$, ${\mathbb E}(\varepsilon^2)=\sigma^2<\infty$ and $\beta\in L^2[0,1]$ is the slope function. The usual aim in such a model is estimating $\beta$ and $\sigma^2$ from an i.i.d. sample $(X_i,Y_i)$, $i=1,\ldots,n$.   As we are assuming ${\mathbb E}(X(t))=0$ we omit as well the additional intercept additive term  $\beta_0$ in the theoretical developments involving model (\ref{eq:l2model}). This term will be incorporated in the numerical examples of Section \ref{sec:empirical}. 

The vast majority of literature on functional linear regression is concerned with model \eqref{eq:l2model}; see, e.g., the pioneering book by \cite{ramsay2005} (whose first edition dates from 1997), as well as   the paper by \cite{cardot1999functional}, the book by \cite{horvath2012} and references therein. Though this model is, in several aspects, natural and useful, we argue here that this is not the only sensible approach to functional linear regression. 

There are several reasons for this statement: first, unlike the finite dimensional model $Y=\beta_1X_1+\ldots+\beta_dX_d+\varepsilon$, there is no obvious, easy to calculate, estimator for $\beta$ under model \eqref{eq:l2model}. The simple, elegant least squares theory is no longer available here. The optimality properties (Gauss-Markov Theorem) of the finite-dimensional least squares estimator do not directly apply for \eqref{eq:l2model} either. Second, note that in the finite-dimensional situation, where $X=(X_1,\ldots,X_d)$,  there is a strong case in favour of a model of type $Y=\beta_1X_1+\ldots+\beta_dX_d+\varepsilon$. The reason is that, as it is well-known, when the joint distribution of all the involved variables is Gaussian, the best approximation of $Y$ in terms of $X=(X_1,\ldots,X_d)$  is necessarily of type $\beta_1X_1+\ldots+\beta_dX_d$. A similar motivating property does not hold for model \eqref{eq:l2model}.
Third, some natural,  linear-like functional approaches do not appear as particular instances of  \eqref{eq:l2model}; this is the case, for example, with an approach of type \textit{``the response $Y$ is (up to an additive noise) a linear combination of a finite or countable subset of variables $\{X(t),\ t\in T\}$''}. 

Our goal here is to analyse a more general linear model which partially addresses these downsides and includes model \eqref{eq:l2model} as a particular case.   Perhaps more importantly, the finite dimensional models of type $Y=\beta_1X(t_1)+\ldots+\beta_pX(t_p)$, with $t_j\in T$, $\beta_j\in {\mathbb R}$ and $p\in {\mathbb N}$ are also included. This is particularly relevant in practice since, in many cases,  such models (obtained, for example, by a variable selection procedure)  have a larger predictive power than that of the $L^2$-based model \eqref{eq:l2model}; see the experiments in Section \ref{sec:empirical}.

It is worth noting that, essentially, this model already appears in the recent paper
by \cite{berrendero2019rkhs} as a suitable setting for variable selection in functional regression. The novel contribution here is to provide a more complete perspective of such model, by formulating it in two equivalent ways and realizing its full generality (beyond variable selection), along the lines commented in the previous paragraph. Also, the estimation of the regression parameter is also explicitly addressed here.  

\noindent \textit{The organization of this paper}

In Section \ref{sec:formulation} our general linear RKHS-based model is defined. Section \ref{sec:particular} is devoted to prove that some relevant examples of practical interest appear just as particular cases of such a model. Two results of consistent estimation of the slope function are given in Section \ref{sec:consistency}. Some experimental results are discussed in 
Section \ref{sec:empirical}.

\section{A general formulation of the functional linear model}\label{sec:formulation}

In the functional framework introduced in the previous section, a linear model might be 
defined as any suitable linear expression of the variables $X(t)$ aiming to explain (predict) the response variable $Y$. The $L^2$-model \eqref{eq:l2model} is just one possible formulation of such idea. 

If the random variables (rv) $X(t)$ are defined in a common sample space $\Omega$, let $L^2(\Omega)$ be the space of all rv's with finite second moment. 
Thus, in the present work, by ``linear expression'' we mean   an element of the closed linear subspace $L_X$ of $L^2(\Omega)$ spanned  by the variables in $\{X(t):\, t\in T\}$. In other words, $L_X$ is  the closure of the linear subspace of all finite linear combinations of variables in the collection.  Hence, $L_X$  includes both finite linear combinations of the form $\sum_{i=1}^p \beta_i X(t_i)$ (where $p\in\mathbb{N}$, $\beta_1,\ldots,\beta_p\in \mathbb{R}$, and $t_i,\ldots,t_p \in T$) and  rv's   $U$ such that there exists  a sequence $U_n$ of these linear combinations with $\|U_n-U\|\to 0$, as $n\to\infty$.

In more precise terms, our general linear model will be defined by assuming that $Y$ and  $\{X(t):\, t\in T\}$ are related by the following linear regression model:
\begin{equation}
	\label{eq.linear-model}
	Y = U_X + \varepsilon,
\end{equation}
where $U_X\in L_X$, and $\varepsilon$ is a random variable with zero mean and variance equal to $\sigma^2$ which is assumed to belong to  $L_X^{\perp}$, the orthogonal complement of $L_X$, that is $ {\mathbb E}(\varepsilon U_X)=0$ for all $U\in L_X$.

Since $L_X$ is closed,  we know  that $L^2(\Omega)=L_X\oplus L_X^{\perp}$, and then the elements in the model are unambiguously given by the orthogonal projections $U_X=\mbox{Proj}_{L_X}(Y)$ and $\varepsilon=\mbox{Proj}_{L_X^{\perp}}(Y).$

\subsection{An RKHS formulation of the proposed linear model}

The aim of this subsection is to give a fairly natural parametrization of model (\ref{eq.linear-model}), based on the RKHS theory. As a consequence of this alternative formulation we will show that several useful linear models appear as particular cases of (\ref{eq.linear-model}).

Denote by $\mathbb{R}^T$ the set of all real functions defined on $T$. 
We are going to define a map $\Psi_X: L_X \to \mathbb{R}^T$  that will play a key role in the sequel: given  $U\in L_X$, $\Psi_X(U)$ is just the function
\begin{equation}
	\label{eq.loeve}
	\Psi_X(U)(t) = \mathbb{E}(UX(t)).
\end{equation}
As we will next show, this transformation defines an isometry (often called Lo\`eve's isometry; see Lemma 1.1 in \cite{lukic2001stochastic}) between $L_X$ and $\Psi_X(L_X)$. We will also see that $\Psi_X(L_X)$ coincides in fact with the RKHS generated by $K$.

To begin with let us recall here, for the sake of completeness, a simple lemma collecting two elementary properties of $\Psi_X$: 
\begin{lemma}
	\label{lemma.elementary}
	Let $\Psi_X(U)$ be as defined in (\ref{eq.loeve}). Then,
	\begin{compactitem}
		\item[(a)]  $\Psi_X$ is injective.
		\item[(b)]  $\Psi_X(X(t))(\cdot)=K(\cdot, t)$, for all $t\in T$. Equivalently,  $\Psi_X^{-1}[K(\cdot,t)]=X(t)$.
	\end{compactitem}
\end{lemma}

\begin{proof}
	To show (a), let $U,V\in L_X$ such that $\Psi_X(U)(t)=\Psi_X(V)(t)$, for all $t\in T$. Then,  $\mathbb{E}[(U-V)X(t)]=0$, for all $t\in T$, what implies $U-V\in L_X^{\perp}$. Since we also have $U-V\in L_X$, we get $U=V$; recall that in $L^2$ spaces we identify those functions that coincide almost surely.
	
	Property (b) is obvious from the definition.
\end{proof}

Denote by $H_K$ the image of $\Psi_X$ so that $\Psi_X: L_X \to H_K$ is a bijection. The subscript $K$ emphasizes the fact that the properties of $H_K$ are closely related to those of the covariance function. Observe that by Lemma \ref{lemma.elementary}(b), all the finite linear combinations $\sum_{j=1}^p \beta_j K(\cdot, t_j)$ belong to  $H_K$. 

The inner product in $L_X$ induces an inner product in  $H_K$: given $f, g\in H_K$, define 
\[
\langle f,g\rangle_K := \langle \Psi_X^{-1}(f),\Psi_X^{-1}(g)\rangle.
\]

As a consequence of the considerations above, $H_K$ endowed with $\langle \cdot, \cdot\rangle_K$, is a Hilbert space. Once we endow $H_K$ with this structure, the mapping $\Psi_X$ is a linear, bijective, and inner product preserving operator between  $L_X$ and $H_K$; this accounts for the word ``isometry'' in the usual name  (which is often called Loève's isometry).

On the other hand, it is well-known (see, e.g., Appendix F in \cite{janson1997gaussian} for details) that, given a positive semi-definite function $K:T\times T\rightarrow {\mathbb R}$ (called ``reproducing kernel''), there is a unique Hilbert space, generated by the linear combinations of the form $\sum_j\beta_jK(\cdot,t_j)$. This space is usually called the Reproducing Kernel Hilbert Space (RKHS) associated with $K$. 

The following simple result shows that $\Psi_X(L_X)=H_K$ coincides in fact with the RKHS generated by the covariance function $K$ of the process $X=X(t)$.  

\begin{proposition}
	The Hilbert space $H_K$ and the covariance function $K$ satisfy the following two properties:
	\begin{compactitem}
		\item[(a)] For all $t\in T$, $K(\cdot, t)\in H_K$. 
		\item[(b)] {\em Reproducing property}: for all $f\in H_K$ and $t\in T$, $\langle f, K(\cdot,t)\rangle_K = f(t)$.
	\end{compactitem}
\end{proposition}

\begin{proof}
	(a) follows directly from Lemma \ref{lemma.elementary}(b). To prove (b),
	\[
	\langle f, K(\cdot,t)\rangle_K = \langle \Psi_X^{-1}(f), X(t)\rangle =\mathbb{E}[\Psi_X^{-1}(f) X(t)] = \Psi_X[\Psi_X^{-1}(f)](t)=f(t).
	\]
	The first equality is due to Lemma \ref{lemma.elementary}(b).
\end{proof}

We are now in a position to recast the general model \eqref{eq.linear-model} into a sort of parametric formulation, where the ``parameter'' belongs to the RKHS generated by the covariance function $K$ of the process $X=X(t)$.
As we will see, this reformulation will be particularly useful to encompass several particular cases of practical relevance.

\begin{theorem}
	Model \eqref{eq.linear-model} can be equivalently established in the form
	\begin{equation}
		\label{eq.linear-loeve}
		Y = \Psi_X^{-1}(\alpha) + \varepsilon,
	\end{equation}
	where $\alpha\in H_K$ and $\varepsilon\in L_X^{\perp}$ is a random variable with zero mean and variance equal to $\sigma^2$ and $\Psi_X^{-1}$ is the inverse of Lo\`eve's isometry $\Psi_X:L_X\rightarrow H_K$ defined above.
	
	In addition the ``parameter'' $\alpha$ is the cross-covariance function 
	$\alpha(t)={\mathbb E}(YX(t))$. 
	
	\begin{proof}
		Formulation \eqref{eq.linear-loeve} follows directly from the definition of $\Psi_X$ and the fact that this transformation is a bijection between $L_X$ and $H_K$; hence $U_X\in L_X$ if and only if there exists a (unique) $\alpha\in H_K$ such that $U_X=\Psi_X^{-1}(\alpha)$. 
		
		To prove the second assertion note that, by the reproducing property, $\alpha(t) = \langle \alpha, K(\cdot,t)\rangle_K $ for all $t\in T$, and hence
		\begin{equation}
			\label{eq.covariance}
			\alpha(t) = \langle \alpha, K(\cdot,t)\rangle_K  =\mathbb{E}[\Psi_X^{-1}(\alpha) X(t)] = \mathbb{E}[(Y-\varepsilon) X(t)] = \mathbb{E}[Y X(t)], 
		\end{equation}
		because $\varepsilon\in L_X^{\perp}$.
	\end{proof}
\end{theorem}

As a consequence, the RKHS $H_K$ is a fairly natural parametric space for our general linear regression model. 

\begin{remark}\label{parzen_notation}
	Let us note that model \eqref{eq.linear-loeve} was already considered, with a different notation, in the paper by \cite{berrendero2019rkhs}. 
	Indeed, the inverse Lo\`eve transformation $\Psi_X^{-1}(\alpha)$ is sometimes denoted $\langle \alpha, X\rangle_K$. This is somewhat of a notational abuse, as typically the trajectories of the process do not belong to the RKHS $H_K$; see \cite{berrendero2019rkhs} for details. Still, the notation is often useful and convenient, so that the RKHS-model can be expressed also by
	\begin{equation}\label{eq:parzen_model}
		Y=\langle X,\alpha\rangle_K+\varepsilon.
	\end{equation}
\end{remark}

\section{Some important particular cases}\label{sec:particular}

The above mentioned work by \cite{berrendero2019rkhs}  focusses in the model \eqref{eq.linear-loeve}-\eqref{eq:parzen_model} from the point of view of its application to variable selection topics. In the present work, we go further in the study of such model: in addition to the relevant equivalence  (that we have just established) between \eqref{eq.linear-loeve} and \eqref{eq.linear-model}, we show in this section the generality of such model, by showing that several commonly used models appear just as particular cases. In Section \ref{sec:consistency} we address the problem of estimating the ``regression parameter'' $\alpha$ and in Section \ref{sec:empirical} we carry out some numerical experiments.

\subsection{Finite dimensional models: a setup for variable selection problems}

When there are infinitely many regressors (which is the case in functional regression problems), several procedures of \textit{variable selection} are available (see \cite{berrendero2019rkhs} for details) for replacing the whole set of explanatory variables with a finite, carefully chosen, subset of these variables.  The following proposition characterizes when it is possible to apply these procedures without any loss of information at all.

\begin{proposition}
	\label{prop.variable-selection}
	Under model \eqref{eq.linear-loeve}-\eqref{eq:parzen_model}, there exist $X(t^*_1),\ldots,X(t_p^*)\in \{X(t):\, t\in T\}$ and $\beta_1,\ldots,\beta_p\in\mathbb{R}$ such that $Y=\beta_1X(t^*_1)+\cdots + \beta_pX(t^*_p)+\varepsilon$ if and only if for all $t\in T$, $\alpha(t) = \beta_1K(t,t_1^*) + \cdots + \beta_pK(t,t_p^*)$.  
\end{proposition}

\begin{proof}
	By Lemma \ref{lemma.elementary}(b), $Y=\beta_1X(t^*_1)+\cdots + \beta_pX(t^*_p)+\varepsilon$ if and only if $\Psi_X^{-1}(\alpha) = \beta_1\Psi_X^{-1}(K(\cdot, t^*_1))+\cdots + \beta_p(K(\cdot, t^*_p))$, what in turn happens if and only if $\alpha(t) = \beta_1K(t,t_1^*) + \cdots + \beta_pK(t,t_p^*)$.
\end{proof}

\subsection{The classical $L^2$-model}

For $T=[0,1]$ assume that $\{X(t):\, t\in T\}$ is an $L^2$ random process and $Y$ a response variable such that the RKHS linear model \eqref{eq.linear-model} or, equivalently \eqref{eq.linear-loeve} or  (\ref{eq:parzen_model}), holds. To gain some insight, let us illustrate this with an example, beyond the finite-dimensional models considered in the previous subsection. 

\noindent {\bf Example (Brownian regressors):} When $\{X(t):\, t\in [0,1]\}$ is a standard Brownian Motion ($K(s,t)=\min\{s,t\}$)  it can be shown  
\[
H_K=\{\alpha\in L^2[0,1]:\, \alpha(0)=0,\ \mbox{there exists}\ \alpha'\in L^2[0,1]\ \mbox{such that}\ \alpha(t) = \int_0^t \alpha'(s)ds\}.
\]
and $\langle \alpha_1,\alpha_2\rangle_K=\langle \alpha'_1,\alpha'_2\rangle_2$.  It can also be proved that  $\Psi_X^{-1}(\alpha)$ is given by Itô's stochastic integral, $\Psi_X^{-1}(\alpha) = \int_0^1\alpha'(t) dX(t)$; for these results see \cite{janson1997gaussian}, Example 8.19, p. 122.  Thus,  in this case, the linear model \eqref{eq.linear-model} and \eqref{eq.linear-loeve} reduces to  $Y= \int_0^1\alpha'(t) dX(t) + \varepsilon$. 

\

Our goal in this subsection is to analyze under which conditions the RKHS model \eqref{eq.linear-loeve} or \eqref{eq:parzen_model} entails the $L^2$-model \eqref{eq:l2model}. To do  this, we need to recall some basic facts about the RKHS space associated with $K$. 
Let us denote by $\mathcal{K}:L^2[0,1]\to L^2[0,1]$ the integral operator defined by $K$, that is
\[
\mathcal{K}f(t) = \int_0^1 K(t,s) f(s) ds.
\]
We will henceforth assume that $K$ is continuous. Under this condition, it is well-known that $\mathcal{K}$ is a compact and self-adjoint operator.  The following proposition is a standard result in the RKHS theory. It relates $H_K$ to $L^2[0,1]$. The proof can be found, e.g. in \cite[Th. 4.1.2]{cucker2001}.

\begin{proposition}
	\label{prop.RKHS-K}
	Assume $T=[0,1]$ and $K$ is continuous. Let $\lambda_1\geq \lambda_2\geq \cdots$ be the eigenvalues of $\mathcal{K}$ (which are assumed to be strictly positive for simplicity) and let $e_1, e_2,\ldots$ be the corresponding unit eigenfunctions. Then, the RKHS corresponding to $K$ is
	\begin{equation*} \label{eq.RKHS-K}
		H_K = \{f\in L^2[0,1]:\,  \sum_{i=1}^\infty \frac{\langle f,e_i\rangle_2^2}{\lambda_i}<\infty\} = \mathcal{K}^{1/2}(L^2[0,1]),
	\end{equation*}
	endowed with the inner product $\langle f,g\rangle_{K}=\sum_{i=1}^\infty \langle f,e_i\rangle_2 \langle g,e_i\rangle_2/\lambda_i$.
\end{proposition}

At this point, we should perhaps recall that according to Spectral Theorem \cite[Ch. 4]{gohberg2004basic} under the stated conditions, there is an orthonormal basis of $L^2[0,1]$ made of eigenvectors $e_i$ whose corresponding eigenvalues are $\lambda_i$, in such a way that, for all $f\in L^2[0,1]$ we have $f=\sum_i\langle f,e_i\rangle e_i$ (where the convergence of this series and the corresponding equality must be understood in the $L^2$ sense) and $\mathcal{K}^{1/2}(f)=\sum_i\sqrt{\lambda_i}\langle f,e_i\rangle e_i$. Thus, the membership to $H_K$ must be understood as a ``regularity property'' established in terms of a very fast convergence to zero of the Fourier coefficients $\langle f,e_i\rangle$.

Now, let us go back to the classical  functional linear regression model \eqref{eq:l2model}. We will show that \eqref{eq:l2model} appears as a particular case of our general model \eqref{eq.linear-loeve} if, and only if, the slope function $\beta$ in \eqref{eq:l2model} is regular enough to belong to the image subspace $\mathcal{K}(L^2[0,1])$ which, by Proposition \ref{prop.RKHS-K}, is a subset of $H_K$. The formal statement is as follows.

\begin{proposition}
	\label{prop.model_L2}
	If model (\ref{eq:l2model}) holds, then model (\ref{eq.linear-loeve}) also holds. Reciprocally, if  (\ref{eq.linear-loeve}) holds and there exists $\beta\in L^2[0,1]$ such that $\alpha = \mathcal{K}\beta$, then (\ref{eq:l2model}) holds.
\end{proposition}

\begin{proof}
	Since $\int_0^1 X(t)\beta(t) dt \in L_X$ (e.g. \cite{ash1975}, p. 34) we have that  (\ref{eq:l2model}) implies (\ref{eq.linear-loeve}).
	Reciprocally, assume $Y=\Psi_X^{-1}(\alpha)+\varepsilon$, where $\alpha = \mathcal{K}\beta$ for $\beta\in L^2[0,1]$, and $\varepsilon\in L_X^{\perp}$. Using Fubini's Theorem we get, for all $t\in [0,1]$,
	\[
	\alpha(t) = \int_0^1 K(t,s)\beta(s) ds = \mathbb{E}\big[X(t) \int_0^1 X(s)\beta(s) ds\big].
	\]
	On the other hand, by (\ref{eq.covariance}) we also have $\alpha(t)= \mathbb{E}[X(t)Y]$, for all $t\in [0,1]$. Then, 
	$\mathbb{E}\big[X(t)\big(Y - \int_0^1 X(t)\beta(t) dt\big)\big] = 0$, for all $t\in [0,1]$. Hence, $Y - \int_0^1 X(t)\beta(t)dt \in L_X^{\perp}$. Now,
	\[
	Y = \int_0^1 X(t)\beta(t) dt + \big(Y - \int_0^1 X(t)\beta(t) dt\big) = \Psi_X^{-1}(\alpha) + \varepsilon,
	\]
	where $\Psi_X^{-1}(\alpha)\in L_X$, $\int_0^1 X(t)\beta(t) dt \in L_X$, and $\varepsilon\in L_X^{\perp}$, $Y - \int_0^1 X(t)\beta(t) dt\in L_X^{\perp}$. Since $L^2(\Omega) = L_X \oplus L_X^{\perp}$, we get $\Psi_X^{-1}(\alpha)=\int_0^1 X(t)\beta(t) dt$ and the result follows.
\end{proof}

Observe that the difference between (\ref{eq.linear-loeve}) and (\ref{eq:l2model}) is not just a minor technical question. There are important values of the parameter $\alpha$ such that $\alpha\in H_K$ but $\alpha \neq \mathcal{K}\beta$. For example, finite linear combinations of the form $\beta_1K(\cdot, t_1)+\cdots + \beta_pK(\cdot,t_p)$, which are important because they allow us to include finite dimensional regression models (also called impact point models in the literature on functional regression)  as particular cases of the general model (see Proposition {\ref{prop.variable-selection} above).

	The  procedures to fit model (\ref{eq:l2model}) very often involve to project $\{X(t):\, t\in [0,1]\}$ on a convenient subspace of $L^2[0,1]$. More precisely, given $\{u_j:\, j = 1,2,\ldots\}$, an arbitrary orthonormal basis of $L^2[0,1]$, it is quite common to use as regressor variables the projections $\{X(t):\, t\in [0,1]\}$ on the finite dimensional  subspace spanned by the first $p$ elements of the basis. This amounts to replace the whole trajectory with $\langle X, u_1\rangle_2 u_1 +\cdots +\langle X, u_p\rangle_2 u_p$. This  method will work fine whenever $\int_0^1 X(t)\beta(t) dt \approx  \sum_{j=1}^p \beta_j \langle X, u_j\rangle_2$, where $\beta_j = \langle \beta, u_j\rangle_2$. A natural question to ask is when there is no loss in using the projection instead of the whole trajectory, and how is this situation characterized in terms of the parameter $\alpha$ in (\ref{eq.linear-loeve}). The answer is given by Proposition \ref{prop.projection} below.
	
	Assume that the following model (explaining the response in terms of the projection of the trajectory) holds:
	\begin{equation}
		\label{eq.projection}
		Y = \beta_1 \langle X, u_1\rangle_2 + \cdots +  \beta_p \langle X, u_p\rangle_2 + \varepsilon,
	\end{equation}
	where $\beta_1,\ldots,\beta_p \in\mathbb{R}$ and $\varepsilon\in L_X^{\perp}$. We have the following result:
	
	\begin{proposition}
		\label{prop.projection}
		If model (\ref{eq.projection}) holds, then model (\ref{eq.linear-loeve}) also holds. Reciprocally, if  (\ref{eq.linear-loeve}) holds and  $\alpha$ belongs to the subspace spanned by $\{\mathcal{K}u_1,\ldots,\mathcal{K}u_p\}$   then (\ref{eq.projection}) holds.
	\end{proposition}
	
	\begin{proof}
		The proof is similar to that of Proposition \ref{prop.model_L2}. For $j=1,\ldots,p$, we have $\langle X, u_j\rangle_2 \in L_X$ (see e.g. \cite{ash1975}, p. 34) and hence  (\ref{eq.projection}) implies (\ref{eq.linear-loeve}).
		
		Reciprocally, assume $Y=\Psi_X^{-1}(\alpha)+\varepsilon$, where $\alpha = \beta_1\mathcal{K}u_1+\ldots + \beta_p\mathcal{K}u_p$ for $\beta_1,\ldots,\beta_p \in\mathbb{R}$, and $\varepsilon\in L_X^{\perp}$. Using Fubini's Theorem:
		\[
		[\Psi_X(\langle X, u_j\rangle_2)](t) =\mathbb{E}[X(t) \langle X, u_j\rangle_2] = \int_0^1 K(t,s) u_j(s) ds = [\mathcal{K}u_j](t). 
		\]
		Then $\Psi_X^{-1}(\alpha) = \beta_1\Psi_X^{-1} (\mathcal{K}u_1) + \cdots +  \beta_p\Psi_X^{-1} (\mathcal{K}u_p) = \beta_1 \langle X, u_1\rangle_2 + \cdots +  \beta_p \langle X, u_p\rangle_2$, 
		and the result follows.
	\end{proof}
	
	An important particular case is functional principal component regression (FPCR). In FPCR, the orthonormal basis is that given by $u_j = e_j$, the eigenfunctions of $\mathcal{K}$. Then, $\mathcal{K}e_j = \lambda_j e_j$, for $j=1,\ldots,p$, and the condition in Proposition \ref{prop.projection} simply states  that $\alpha$ must belong to the subspace spanned by $\{e_1,\ldots,e_p\}$.
	
	\section{Consistent estimation of the slope function} \label{sec:consistency}
	
	\subsection{An $L^2$-consistent estimator based on regularization}
	
	Let $(Y_1,X_1),\ldots,(Y_n,X_n)$ is a sample of i.i.d. observations following the functional regression model (\ref{eq.linear-loeve}). In this section we give a consistent estimator of $\alpha$ based on Tikhonov regularization.

	The interpretation of $\alpha$ as a covariance given by Equation (\ref{eq.covariance}) suggests a natural way to estimate it. We could just use the sample covariance function,
	\[
	\tilde{\alpha}(t) = \frac{1}{n}\sum_{i=1}^n Y_iX_i(t).
	\] 
	Unfortunately, $\mathbb{P}(\tilde{\alpha} \in H_K)=0$ (see \cite{lukic2001stochastic}) whereas we are assuming $\alpha \in H_K$. Then, we need to apply a  regularization technique in order to take $\tilde{\alpha}$ into the RKHS. A possibility is to use Tikhonov regularization, which leads to the following estimator of $\alpha$:
	\[
	\check{\alpha} = \arg\min_{f\in H_K} \|\tilde{\alpha} - f\|^2_2 +  \gamma_n \|f\|^2_{K},
	\]
	where $\|\cdot\|_2$ denotes the norm in $L^2[0,1]$, $\|\cdot\|_K$ denotes the norm in $H_K$ and $\gamma_n>0$  is a sequence of regularization parameters depending on the sample size. It turns out that $\check{\alpha}$ has the following explicit expression: 
	\begin{equation}\label{plugin}
		\check{\alpha} = (\mathcal{K}+\gamma_n I)^{-1}\mathcal{K}\tilde{\alpha}, 
	\end{equation}
	where $\mathcal{K}$ is the integral operator defined by the kernel $K$;   see \citet[p. 139]{cucker2007learning}.  Note that \eqref{plugin} is a sort of ``oracle estimator'' since, in practice, $K$ is generally unknown so that $\mathcal{K}$  must be also estimated from the sample.   This leads to the  final estimator
	\begin{equation}\label{final}
		\hat{\alpha} := (\widehat{\mathcal{K}}+\gamma_n I)^{-1}\widehat{\mathcal{K}}\tilde{\alpha}, 
	\end{equation}
	where $\widehat{\mathcal{K}}$ is the integral operator defined by the kernel $\hat{K}(s,t)=n^{-1}\sum_{i=1}^n X_i(s)X_i(t)$.

	The consistency   of $\hat\alpha$ in $L^2$ norm  is stated in the following result:
	
	\begin{theorem}
		\label{th.consistency-tikhonov}
		Let $\gamma_n\to 0$ be  such that $\gamma_n^2\sqrt{n}\to \infty$ and assume that $\mathbb{E}(\|X\|_2^4) < \infty$. Then 
		$\|\hat{\alpha}-\alpha\|_2\to 0$ in probability, as $n\to \infty.$
		
	\end{theorem}
	
	\begin{proof}
		Let us define  $\mathcal{T}_n:=(\mathcal{K}+\gamma_n I)^{-1}\mathcal{K}$ and $\widehat{\mathcal{T}}_n:=(\widehat{\mathcal{K}}+\gamma_n I)^{-1}\widehat{\mathcal{K}}$. Then $\| \hat{\alpha} - \alpha\|_2 \leq \|\hat{\mathcal{T}_n}\tilde{\alpha} - \mathcal{T}_n\alpha\|_2 +  \|\mathcal{T}_n\alpha - \alpha\|_2$. We analyse both terms separately.
		
		\noindent{\em First term.} For the first term, $\|\hat{\mathcal{T}}_n\tilde{\alpha} - \mathcal{T}_n\alpha\|_2 \leq 
		\|\hat{\mathcal{T}}_n\tilde{\alpha} - \hat{\mathcal{T}}_n\alpha\|_2 + \|\hat{\mathcal{T}}_n\alpha - \mathcal{T}_n\alpha\|_2 \leq \|\hat{\mathcal{T}}_n\|_{op} \|\tilde{\alpha} -\alpha\|_2 + \|\hat{\mathcal{T}}_n-\mathcal{T}\|_{op} \|\alpha\|_2,$
		where $\|\cdot\|_{op}$ stands for the usual operator norm for linear continuous operators, $\|Tx\|_{op}=\sup_{\|x\|=1}\|Tx\|$. Now it is enough to prove $\|\tilde{\alpha} -\alpha\|_2 \to 0$ and $\|\hat{\mathcal{T}}_n-\mathcal{T}\|_{op}\to 0$  in probability.
		
		Convergence $\|\tilde{\alpha} -\alpha\|_2 \to 0$   in probability follows from Mourier's SLLN (see e.g. Theorem 4.5.2 in \cite{lahaprobability} p. 452). To prove $\|\hat{\mathcal{T}}_n-\mathcal{T}\|_{op}\to 0$   in probability, use triangle inequality to bound 
		$\|\hat{\mathcal{T}}_n-\mathcal{T}_n\|_{op}\leq \|(\hat{\mathcal{K}}+\gamma_n I)^{-1}\|_{op}\|\|\hat{\mathcal{K}}-\mathcal{K}\|_{op}+\|(\hat{\mathcal{K}}+\gamma_n I)^{-1}-(\mathcal{K}+\gamma_n I)^{-1}\|_{op}\|\mathcal{K}\|_{op}.$ 
		From Equation (1.14) in \cite{gohberg2004basic} p. 228,
		\begin{equation}
			\label{eq.b3}
			\|(\hat{\mathcal{K}}+\gamma_n I)^{-1}\|_{op}\leq \gamma_n^{-1},\  \ \mbox{for all}\ n > 0.
		\end{equation}
		Whenever $\|A^{-1}\|_{op}\|A-B\|_{op}<1$ the following inequality holds (\cite{gohberg2004basic}, p. 71)
		$\|A^{-1}-B^{-1}\|_{op}\leq (\|A^{-1}\|_{op}^2\|A-B\|_{op})(1-\|A^{-1}\|_{op} \|A-B\|_{op})^{-1}.$
		From this inequality, together with \eqref{eq.b3}, we get 
		\[
		\|(\hat{\mathcal{K}}+\gamma_n I)^{-1}-(\mathcal{K}+\gamma_n I)^{-1}\|_{op}\leq \frac{(1/\gamma_n)^2\|\hat{\mathcal{K}}-\mathcal{K}\|_{op}}{1-(1/\gamma_n)\|\hat{\mathcal{K}}-\mathcal{K}\|_{op}}.
		\]
		To prove that $(1/\gamma_n)^2\|\hat{\mathcal{K}}-\mathcal{K}\|_{op}\to 0$ in probability we will use that $\|\hat{\mathcal{K}}-\mathcal{K}\|_{op}\leq \|\hat{\mathcal{K}}-\mathcal{K}\|_{HS}$,   where $\|\cdot\|$ stands for the Hilbert-Schmidt norm for operators.
		Theorem 8.1.2 in \cite{hsing2015theoretical} p. 212, guarantees that if $\mathbb{E}(\|X\|_2^4)<\infty$, then $\sqrt{n}\|\hat{\mathcal{K}}-\mathcal{K}\|_{HS}$ is bounded in probability. Then, for all  $\epsilon>0$ there exists $M$ such that for $n$ large enough, $\mathbb{P}(\sqrt{n}\|\hat{\mathcal{K}}-\mathcal{K}\|_{HS}>M)<\epsilon$.
		Now, let us fix an arbitrary $\delta>0$.  Let $n$ be large enough such that $\sqrt{n}\gamma_n^2\delta>M$ (remember that we are assuming $\sqrt{n}\gamma_n^2\to\infty$). Then, 
		\[
		\mathbb{P}((1/\gamma_n)^2\|\hat{\mathcal{K}}-\mathcal{K}\|_{op}>\delta)\leq  \mathbb{P}(\sqrt{n}\|\hat{\mathcal{K}}-\mathcal{K}\|_{HS}>M)<\epsilon.
		\]
		
		\noindent{\em Second term.} The term  $\|\mathcal{T}_n\alpha - \alpha\|_2$ will be small for small $\gamma_n$. Then, we need to consider a sequence  $\gamma_n \downarrow 0$, but we must take into account that $\mathcal{K}$ is not invertible.
		
		Let $\{e_j\}_j$ be an orthonormal base of eigenvectors of $\mathcal{K}$ with associated eigenvalues $\{\lambda_j\}_j$. Let $\epsilon>0$ and $N=N(\epsilon)$ such that $\sum_{j=N+1}^\infty \langle \alpha,e_j\rangle_2^2<\epsilon/2$.
		Observe that 
		\[
		\mathcal{T}_n\alpha=\sum_{j=1}^\infty \frac{\lambda_j}{\gamma_n+\lambda_j} \langle \alpha, e_j\rangle_2 e_j\ \ \mbox{and}\ \ \sum_{j=N+1}^\infty \frac{\lambda_j^2}{(\gamma_n+\lambda_j)^2} \langle \alpha, e_j\rangle_2 ^2 < \epsilon/2.
		\] 
		Then, for large enough $n$, $\|\mathcal{T}_n\alpha-\alpha\|_2^2=\sum_{j=1}^N (\frac{\lambda_j}{\gamma_n+\lambda_j}-1)^2\langle \alpha,e_j  \rangle_2^2+\epsilon/2 < \epsilon.$ 
		Hence, $\|\mathcal{T}_n\alpha-\alpha\|_2^2\to 0$ as $n\to \infty$.
	\end{proof}
	
	We now establish the consistency of the simple plug-in estimator defined in \eqref{plugin}  in the RKHS norm. Note that this estimator requires the knowledge of the true covariance operator, which is seldom the case in practice. Still, the following result has perhaps some theoretical interest; see the discussion at the beginning of subsection \ref{subsec:discret}. 
	
	\begin{proposition}\label{prop:convrkhs} Let $\gamma_n\to 0$ such that  $n\gamma^2_n\to \infty$, then $||\check{\alpha}-\alpha||_K^2\to 0$ in probability.
	\end{proposition}
	
	\begin{proof}  To prove that $\|\check{\alpha}-\alpha\|_K\to0$ in probability, observe that since $\alpha\in H_K$, for all $\epsilon>0$ there exists $N=N(\epsilon)$ such that
		\begin{equation}
			\label{eq1}
			\sum_{j=N+1}^\infty \frac{1}{\lambda_j}\langle \alpha,e_j\rangle_2^2<\epsilon.
		\end{equation}
		
		For this value of $N$ we have
		\begin{multline}
			\label{bound}
			\|\check{\alpha}-\alpha\|_K \leq \Big\|\sum_{j=1}^N \Big(\frac{\lambda_j}{\gamma_n+\lambda_j}-1\Big)\langle\tilde{\alpha},e_j\rangle_2e_j\Big\|_K+\Big\|\sum_{j=1}^N \langle\tilde{\alpha}-\alpha,e_j\rangle_2e_j\Big\|_K\\+\Big\|\sum_{j=N+1}^\infty \frac{\lambda_j}{\gamma_n+\lambda_j}\langle\tilde{\alpha}-\alpha,e_j\rangle_2 e_j\Big\|_K
			+\Big\|\sum_{j=N+1}^\infty\Big(\frac{\lambda_j}{\gamma_n+\lambda_j}-1\Big)\langle\alpha,e_j\rangle_2 e_j\Big\|_K.
		\end{multline}
		We will look at each term in the expression above. For the first one, we have:
		\begin{multline*}
			\Big\|\sum_{j=1}^N \Big(\frac{\lambda_j}{\gamma_n+\lambda_j}-1\Big)\langle\tilde{\alpha},e_j\rangle_2e_j\Big\|_K\leq 
			\Big\|\sum_{j=1}^N \Big(\frac{\lambda_j}{\gamma_n+\lambda_j}-1\Big)\langle\tilde{\alpha}-\alpha,e_j\rangle_2e_j\Big\|_K+\\
			\Big\|\sum_{j=1}^N \Big(\frac{\lambda_j}{\gamma_n+\lambda_j}-1\Big)\langle\alpha,e_j\rangle_2e_j\Big\|_K.
		\end{multline*}
		Now, observe that $\|\tilde{\alpha}-\alpha\|_2\to 0$ a.s., and let us define
		\[
		C_{N,n} := \max_{j=1,\ldots,N} \Big(\frac{\lambda_j}{\gamma_n+\lambda_j}-1\Big)^2 \|e_j\|_K^2 = \Big(\frac{\lambda_1}{\gamma_n+\lambda_1}-1\Big)^2 \frac{1}{\lambda_N} \to 0,\ \ \mbox{as}\ \ n\to\infty.
		\] 
		Then, for large enough $n$, with probability one,
		\[
		\Big\|\sum_{j=1}^N \Big(\frac{\lambda_j}{\gamma_n+\lambda_j}-1\Big)\langle\tilde{\alpha}-\alpha,e_j\rangle_2e_j\Big\|^2_K\leq N C_{N,n}\|\tilde{\alpha}-\alpha\|^2_2 <\epsilon.
		\]
		We have used Cauchy-Schwarz inequality in the first inequality above. Similarly, we also have, for large enough $n$,
		\[
		\Big\|\sum_{j=1}^N \Big(\frac{\lambda_j}{\gamma_n+\lambda_j}-1\Big)\langle\alpha,e_j\rangle_2e_j\Big\|_K\leq N C_{N,n}\|\alpha\|^2_2 <\epsilon,
		\]
		The second term in (\ref{bound}) satisfies $\|\sum_{j=1}^N \langle\tilde{\alpha}-\alpha,e_j\rangle_2e_j\|^2_K \leq N\lambda_N^{-1}\|\tilde{\alpha}-\alpha\|^2_2<\epsilon,$ for large enough $n$, with probability one.
		For the third term in (\ref{bound}), let $\sum_j \lambda_{j=1}^\infty:=C<\infty$. Then,
		\begin{multline*}
			\Big\|\sum_{j=N+1}^\infty \frac{\lambda_j}{\gamma_n+\lambda_j}\langle\tilde{\alpha}-\alpha,e_j\rangle_2 e_j\Big\|^2_K=
			\sum_{j=N+1}^\infty \frac{\lambda_j}{(\gamma_n+\lambda_j)^2}\langle\tilde{\alpha}-\alpha,e_j\rangle_2^2\\
			\leq \frac{||\tilde{\alpha}-\alpha||_2^2}{\gamma_n^2}\sum_{j=N+1}^\infty \lambda_j \leq C\frac{n||\tilde{\alpha}-\alpha||_2^2}{n\gamma_n^2}\to 0,
		\end{multline*}
		in probability, since we are assuming $n\gamma_n^2\to\infty$ and  $n\|\tilde{\alpha}-\alpha\|_2^2$ is bounded in probability by the Central Limit Theorem. Finally, the fourth term in (\ref{bound}) is also bounded by $\epsilon$ using \eqref{eq1}:
		\[
		\Big\|\sum_{j=N+1}^\infty\Big(\frac{\lambda_j}{\gamma_n+\lambda_j}-1\Big)\langle\alpha,e_j\rangle_2 e_j\Big\|^2_K=\sum_{j=N+1}^\infty 
		\frac{1}{\lambda_j}\Big(\frac{\gamma_n}{\gamma_n+\lambda_j}\Big)^2\langle\alpha,e_j\rangle_2^2\leq \sum_{j=N+1}^\infty 
		\frac{1}{\lambda_j}\langle\alpha,e_j\rangle_2^2<\epsilon.
		\]

	\end{proof}
	
	\subsection{An RKHS-consistent estimator based on discretization} \label{subsec:discret}
	
	The analysis of the proof of Proposition \ref{prop:convrkhs},  concerning consistency of the ``oracle estimator'' $\check{\alpha}$,   suggests an obvious line of attack for the problem of establishing   the RKHS consistency for the more realistic estimator $\hat\alpha$ defined in \eqref{final}: one could think of just replacing the eigenvalues and eigenfunctions, $\lambda_i$ and $e_i$, with the natural empirical estimators. Nevertheless such strategy suffers from a serious practical problem: the eigenvalues $\lambda_i$ appear in the denominator of the RKHS norm; so, as the sequence of eigenvalues tend to zero, a very small error in the estimation of the $\lambda_i$ could have a huge effect in the estimation of the norm. As a consequence, the convergence condition (\ref{eq1}) looks quite difficult to ensure from an analogous inequality based on the estimated $\lambda_i$ and $e_i$. 
	
	Therefore, we will try here a different approach: instead of approximating the eigen-structure we will rather consider a discrete approximation of the linear model itself, taking advantage of the RKHS structure (unlike the ``direct'' approach based on the estimation of the $\lambda_i$ and $e_i$). More specifically, we will approximate the RKHS model \eqref{eq.linear-loeve} by a sequence of finite dimensional models, of type of those considered in Proposition \ref{prop.variable-selection}, based on  $p_n$-dimensional marginals $(X(t_{1,p}),\ldots,X(t_{p,p}))$, obtained by evaluating the process $X(t)$ at the grid points $T_p=\{t_{i,p}\}$, where $p=p_n$. The corresponding sequence of least squares estimators of the slope function $\hat \alpha_p$ will hopefully provide a consistent sequence of estimators of the true slope function $\alpha$ in \eqref{eq.linear-loeve}. This idea is next formalized.

	We will use the following lemma (which follows from Theorem 6E in \cite{parzen59}), that states that function $\alpha$ can be approximated (in the RKHS norm) by a finite linear combination of the kernel function $K$, evaluated at points of a partition of $[0,1]$.

	\begin{lemma} \label{parzen} Let $\alpha\in H_K$. Let us consider $T_p=\{t_{j,p}:\, j=1,\ldots,p\}$ where $0\leq t_{1,p}\leq \dots\leq t_{p,p}\leq 1$, an increasing sequence of partitions of $[0,1]$, i.e,  $T_p\subset T_{p+1}$, such that $\overline{\cup_pT_p}=[0,1]$. Then, there exist $\beta_{1,p}\dots,\beta_{p,p}$ such that,
		$\left\|\alpha(\cdot)-\sum_{j=1}^p \beta_{j,p} K(t_{j,p},\cdot)\right\|_K^2\rightarrow 0,\ \mbox{as } p\to\infty.$ 
	\end{lemma}
	
	Now our estimator is defined by the ordinary least squares estimator of the coefficients $\beta_{1,p},\dots,\beta_{p,p}$. To be more precise, let us denote
	\begin{equation}\label{eq.alphap}
		\alpha_{p}(\cdot) =\sum_{j=1}^{p}\beta_{j,p} K(t_{j,p},\cdot) \mbox{ and } \hat{\alpha}_p(\cdot)=\sum_{j=1}^{p} \hat{\beta}_{j,p}K(t_{j,p},\cdot),
	\end{equation}
	where $t_{1,p},\dots,t_{p,p}$ are chosen as indicated in Lemma \ref{parzen} and, for $j=1,\dots,p$, $\hat{\beta}_{j,p}$ are the ordinary least squares estimators (based on a sample of size $n$) of the regression coefficients in the $p$-dimensional  linear model
	 \begin{equation}\label{approx}
		Y_i=\sum_{j=1}^{p} \beta_{j,p} X(t_{j,p})+e_{i,p}=\langle \alpha_{p},X\rangle_K +e_{i,p},\ \ \ i=1,\ldots,n.
	\end{equation} \color{black}

	To prove the almost sure consistency of the estimator we will need to impose a condition of sub-Gaussianity.  Let us recall that a random variable $Y$ with ${\mathbb E}(Y)=0$ is said to be sub-Gaussian with (positive) proxy constant $\sigma^2$ (we will denote $Y\in SG(\sigma^2)$) if the moment generating function of $Y$ satisfies ${\mathbb E}\left(\exp(sY)\right)\leq \exp(\sigma^2s^2/2)$, for all $s\in {\mathbb R}$. It can be seen that the tails of a random variable $Y\in SG(\sigma^2)$ are lighter than or equal to those of a Gaussian distribution with variance $\sigma^2$, \textit{i.e.}  ${\mathbb P}(|Y|>t)\leq 2\exp(-t^2/(2\sigma^2))$ for all $t>0$. A $p$-dimensional random vector $X$ is said to be sub-Gaussian with proxy constant $\sigma^2$ if $X^\prime v\in SG(\sigma^2)$ for all $v\in {\mathbb R}^d$ with $\|v\|=1$. Observe that if $X$ is a $p$-dimensional random variable $X=(X_1,\ldots,X_p)$ and the $X_i$ are independent  with $X_i\in SG(\sigma^2)$ and sub-Gaussian then $X$ is sub-Gaussian with proxy constant $\sigma^2$ as well. 
	See \cite[Ch. 1]{rigollet2017} for details.

	\begin{theorem} \label{rkhsnorm} Assume the RKHS-based linear model  	$Y_i=\langle X,\alpha\rangle_K+ \epsilon_i$   for $i=1,\ldots,n$, as defined in \eqref{eq.linear-model}, \eqref{eq.linear-loeve} or \eqref{eq:parzen_model}.   Let  us consider a sequence of approximating $p$-dimensional models (with $p=p_n$) as defined in \eqref{approx}. Assume that  
		
		\begin{itemize}
			\item[(i)]    The error variables $e_{i,p}=e_i$ in the p-dimensional models are iid and sub-Gaussian, $SG(\sigma^2_p)$ with $ \sigma^2_p\geq C_0$ for all $p$ and some $C_0>0$. 
			
			\item[(ii)] The random variable  $\sup_{t\in [0,1]} X(t)$ is  sub-Gaussian as well.
			
			\item[(iii)] We have $p\to\infty$, as $n\to\infty$, in such a way that $n(\gamma_{p,p})^2/(p^2\log^3 n)\to \infty$, where $\gamma_{p,p}$ is the smallest eigenvalue of the covariance matrix $K_{T_p}$, of $(X(t_{1,p}),\ldots,X(t_{p,p}))$.
		\end{itemize}
		
		Then,  $\nu_n\|\hat{\alpha}_p-\alpha_p\|_K^2\to 0$ almost surely (a.s.), for all $\nu_n\to \infty$ such that
		$n\gamma_{p,p}/(p^2\nu_n\log n)\to \infty$. In addition, as a consequence of Lemma \ref{parzen},
		$\|\alpha-\hat{\alpha}_p\|_K^2=\max\{\nu_n^{-1},\mathcal{O}(\|\alpha-\alpha_p\|_K^2)\}\  \mbox{a.s.}$
	\end{theorem}
	The proof of this theorem is deferred to the appendix as it is a bit more technical than those of the previous results in the paper. 
	Hypothesis (ii) in Theorem \ref{rkhsnorm} is satisfied for the case of stochastic processes with stationary and independent increments and equispaced impact points $t_{j,p}$, as it is stated in the following proposition.

	\begin{proposition} \label{levy} Let $\{W(t)\}_{t\in [0,1]}$ be a stochastic process with stationary and independent increments, such that $\mathbb{E}(W^2(t))<\infty$ and $\mathbb{E}(W(t))=0$ for all $t\in [0,1]$,  then for all $\delta>0$, $p^{1+\delta}\gamma_{p,p}\to \infty$, $\gamma_{p,p}$being  the smallest eigenvalue of $K_{T_p}$, the covariance matrix of the random vector $(W(1/p),\dots,W(1))$.
		
	\end{proposition}
	
	\begin{proof}
		Let us denote, with some notational abuse, $W=(W(1/p),\dots,W(1))$, $t_i=i/p$, and $v=(v_1,\dots,v_p)$.	Let us introduce the $p\times p$ matrix $A$, such that $WA$ is the $1\times p$ row vector
		$(W(1/p),W(2/p)-W(1/p),\dots , W(1)-W(1-1/p))$, that is $A=(a_{ij})$ where $a_{ii}=1$ for $i=1,\dots,p$, $a_{i-1,i}=-1$ for $i=2,\dots,p$ and $a_{ij}=0$ otherwise.
		The coordinates of $WA$ are independent random variables. Then, for all $v$, 
		$$\mathbb{E}|WAv|^2=v^\prime A^\prime \mathbb{E}(W^\prime W)Av= v^\prime A^\prime K_{T_p}Av= \|Av\|_2^2\frac{v^\prime A^\prime}{\|v'A'\|_2} K_{T_p}\frac{Av}{\|Av\|_2}.$$
		Since $A$ is invertible there exists $v$ with $\|v\|_2=1$ such that  $w:=Av$ fulfills $K_{T_p}w=\gamma_{p,p}w$
		\begin{multline*}
			\min_{v:\|v\|_2=1} \|Av\|_2^2\frac{v^\prime A^\prime}{\|v'A'\|_2} K_{T_p}\frac{Av}{\|Av\|_2}\leq  \|A\|_{op}^2\min_{v:\|v\|_2=1} \frac{v^\prime A^\prime}{\|v'A'\|_2} K_{T_p}\frac{Av}{\|Av\|_2}\leq \|A\|^2_{op}\frac{w^\prime}{\|w\|_2}K_{T_p}\frac{w}{\|w\|_2}\\=\|A\|_{op}^2 \gamma_{p,p}.
		\end{multline*}
		
		Then $\gamma_{p,p}\geq \|A\|_{op}^{-2} \min_{v:\|v\|_2=1}\mathbb{E}|WAv|^2.$ 
		Since $Var(W_{t+s}-W_t)=\sigma^2s$ for all $0\leq t,s\leq 1$, such that $s+t\leq 1$, and for some $\sigma>0$, then, if $\|v\|_2=1$, it follows that 
		$$\mathbb{E}|WAv|^2=\sigma^2\sum_{j=0}^{p-1} \mathbb{E}\big(W((j+1)/p)-W(j/p)\big)^2 v_{j+1}^2=\frac{\sigma^2}{p}$$
		Then $\gamma_{p,p}\geq \sigma^2/(p\|A\|^2_{op})$. Lastly, $\|A\|^2_{op}=1/p+ (4/p)(p-1)$, (the maximum of $\|Av\|$ conditioned to $\|v\|=1$ is when $v_i=(-1)^{i+1}/\sqrt{p}$).
	\end{proof}

	\begin{remark} In \cite{gj08} it is proved that (see eq. (68))   the fractional Brownian Motion with Hurst exponent $H$ fulfils that $\gamma_{p,p}=\mathcal{O}(1/p^{2H})$.
	\end{remark}
	

	\section{Some empirical results}\label{sec:empirical}
	
	We will consider here different examples of functional regression problems in which the goal is to predict a real random variable $Y$ from a functional explanatory variable $X=X(t)$, where $t$ ranges in some known interval $I$. Hence, our sample information is given by $n$ pairs $(X_i(\cdot),Y_i)$, $i=1,\ldots,n$, where $X_i(\cdot)=X_i(t)$ are sample trajectories of the process $X(t)$ and $Y_i$ are the corresponding response variables.
	
	The overall aim of this section is to check the performance of different finite-dimensional models, based on a few one-dimensional marginals $X(t_1),\ldots,X(t_p)$,   such as those whose asymptotic behavior has been analyzed in the previous section,  versus that of a functional $L^2$-based counterpart. 
	More precisely, we will compare the performance of a model of type
	\begin{equation}
		Y=\beta_0+\beta_1X(t_1)+\ldots+\beta_pX(t_p)+\varepsilon, \label{eq:finite}
	\end{equation}
	with that of 
	\begin{equation}
		Y=\beta_0+\int_0^1\beta(t)X(t)dt+\epsilon.\label{eq:L2intercept}
	\end{equation} 
	The word ``performance'' must be understood in terms of ``prediction capacity'', as measured by appropriate estimations of the prediction error ${\mathbb E}[(\hat Y-Y)^2]$, $\hat Y$ being the predicted value for the response obtained from the fitted model.  
	
	It is very important to note that the finite-dimensional models of type \eqref{eq:finite} are viewed here as functional models, in the sense that they are all considered as particular cases of the RKHS-model \eqref{eq.linear-loeve}. In practice, this means that we don't assume any prior knowledge about the ``impact points'' $t_i$ or the number $p$ of variables. So, in principle, the whole trajectory is available in order to pick up the impact points $t_i$ we will use. However, given the grid points $t_i$, the model \eqref{eq:finite} is handled as a problem of multivariate regression.

	\subsection{Simulation experiments}

	\noindent \textit{The models we use to generate the data}
	
	We analyse here three scenarios: one of them is more or less ``neutral'' in the comparison finite-dimensional vs. $L^2$-model. The second one is somewhat favorable to the finite-dimensional models, in the sense that one of these models is the ``true one'', though we have no advance knowledge about the impact points $t_i$ and the number of them.  Finally, the third scenario clearly favours the $L^2$-choice since the data are generated according to a model of type $\eqref{eq:L2intercept}$. We now define these scenarios in precise terms.
	
	\begin{itemize}
		\item[\textbf{Scenario 1.}] We use the function \tt rproc2fdata \rm of the R-package \tt fda.usc \rm \citep{fda.usc} to generate random trajectories according to a fractional Brownian Motion (fBM) $X(t)$ and the aim is to predict $Y=X(1)$ from the observation of the sample trajectories $X(t)$ for $t\in[0,0.95]$.   Let us recall that the fBM is a Gaussian process whose covariance function is 
		$K(s,t)=0.5(|t|^{2H}+|s|^{2H}-|t-s|^{2H})$, $H$ being the so-called ``Hurst exponent''. We have taken $H=0.8$.
		\item[\textbf{Scenario 2.}] We have generated the responses $Y_i$ according to the following two finite-dimensional models (previously considered in \cite{berrendero2019rkhs}),
		
		\textbf{Model 2a:} $Y=2X(0.2)-5X(0.4)+X(0.9)+\varepsilon.$ \\
		\textbf{Model 2b:} $Y=2.1X(0.16)-0.2X(0.47)-1.9X(0.67)+5X(0.85)+4.2X(0.91)+\varepsilon,$ 
		where in both cases the error variable $\varepsilon$ has a distribution $N(0,\sigma)$ with standard deviation $\sigma=0.2$. The process $X(t), t\in [0,1]$ follows a centered fBM with $H=0.8$. 
		\item[\textbf{Scenario 3.}] The response variable $Y$ is generated according to a $L^2$-based linear model with
		\begin{equation*}\label{3linear}
			Y=\int_0^1\log(1+4s)X(s)ds+\varepsilon,
		\end{equation*}
		where, again, the trajectories $X(t)$ are drawn from the same fBM indicated above and $\varepsilon$ has a $N(0,\sigma =0.2)$ distribution.  
	\end{itemize}
	Note that these models are only used to generate the data, so that none of the regression models we will compare in our simulations below will incorporate any prior knowledge on the true distribution of $(X(t), Y)$ whatsoever.

	\noindent \textit{The specific regression models and estimation methods we compare}
	
	Let us go back to our basic question: to what extent the finite-dimensional models (based on marginals $X(t_i)$) of type \eqref{eq:finite} are competitive against a standard, $L^2$-based regression model of type \eqref{eq:L2intercept}? 
	Since we don't assume any previous knowledge on the underlying models generating the data, we will take the ``impact points'' $t_1,\ldots,t_p$ equispaced in the observation interval $[0,1]$ (or, in the interval $[0,0.95]$ in Scenario 1 above). The coefficients $\beta_i$ in this model are estimated by the ordinary least squares method for multiple regression, using the R-function \tt lm\rm. We will check several values of $p=6, 10, 14, 18$. 
	
	As for the $L^2$-regression model \eqref{eq:L2intercept}, we will estimate the slope function $\beta$ and the intercept $\beta_0$ by the so-called functional principal components (PC) method; this essentially amounts to approximate the model \eqref{eq:L2intercept} with another finite-dimensional model obtained by projecting the functional data on a given number $q$ of principal functions (i.e., eigenfunctions of the covariance operator). We use the function \tt fregre.pc \rm  of the R-package \tt fda.usc.\rm Thus, the notation $L^2_q$ will refer to the use of this PC-based estimation method, with $q$ eigenfunctions, for the model \eqref{eq:L2intercept}.
	The considered sample sizes are $n=100, 300, 500, 700$.

	\noindent \textit{The simulation outputs}
	
	  In the left panels of the tables below we report, under the different scenarios, the mean over 100 replications of the standard error, $\sqrt{(1/k)\sum(Y_i-\hat{Y}_i)^2}$, where $k=0.2 \,n$ is the size of the random ``test sub-samples'' we use to evaluate the predictions $\hat Y_i$. . The ``training subsamples'', made of the remaining 80\% of  observations, are used to estimate the regression coefficients for $\hat Y_i$.   The right panels provide the adjusted coefficients of determination given by  $R^2_a=1-(1-R^2)(n-1)/(n-p-1)$, where $R^2=\sum_i(\hat Y_i-Y_i)^2/\sum_i(Y_i-\bar Y)^2$ is the ordinary coefficient of determination. For the $L^2_q$-functional models we replace $p$ in $R^2_a$ with the number $q$ of principal components used in the fit. \color{black}
	
	\
	
	\begin{table}[h]
		\centering\footnotesize
		\begin{tabular}{|r|rrrr|}
			\hline
			\diagbox[height=1\line]{$p$}{ $n$}	&     100 &     300 &     500 &     700 \\ \hline
			6 & 0.23838 & 0.23670 & 0.23524 & 0.23400 \\
			10 & 0.18485 & 0.17618 & 0.17338 & 0.17370 \\
			14 & 0.16162 & 0.15017 & 0.14909 & 0.14929 \\
			18 & 0.14811 & 0.13475 & 0.13390 & 0.13310 \\ \hline
			$L^2_4$ & 0.13989 & 0.13752 & 0.13736 & 0.13788 \\
			$L^2_6$ & 0.12326 & 0.11841 & 0.11767 & 0.11859 \\ \hline
		\end{tabular}\hspace{.4cm}
		\begin{tabular}{|rrrr|}
			\hline
			100 &     300 &     500 & 700     \\ \hline
			0.94359 & 0.94364 & 0.94394 & 0.94407 \\
			0.96977 & 0.97034 & 0.96993 & 0.96982 \\
			0.97830 & 0.97865 & 0.97839 & 0.97803 \\
			0.98251 & 0.98295 & 0.98277 & 0.98250 \\ \hline
			0.98098 & 0.98125 & 0.98077 & 0.98088 \\
			0.98587 & 0.98631 & 0.98615 & 0.98600 \\
			\hline
		\end{tabular}
		\caption{\footnotesize Prediction errors (left) and $R_a^2$ values (right) for the different models under \textbf{Scenario 1}.}
		\label{Sc1}
	\end{table}
	
	\newpage

	\begin{table}[h]
		\centering\footnotesize
		\begin{tabular}{|r|rrrr|}
			\hline
			\diagbox[height=1.1\line]{$p$}{ $n$}	&     100 &     300 &     500 &     700 \\ \hline
			6 & 0.41801 & 0.40521 & 0.40834 & 0.41084 \\
			10 & 0.22031 & 0.21508 & 0.21211 & 0.21196 \\
			14 & 0.28902 & 0.27278 & 0.26997 & 0.26887 \\
			18 & 0.24875 & 0.23343 & 0.23137 & 0.22897 \\ \hline
			$L^2_4$ & 0.37890 & 0.37307 & 0.37242 & 0.37872 \\
			$L^2_6$ & 0.31213 & 0.31609 & 0.31690 & 0.31626 \\ \hline
		\end{tabular} \hspace{.2cm}
		\begin{tabular}{|rrrr|}
			\hline
			100 &     300 &     500 & 700     \\ \hline
			0.90024 & 0.89827 & 0.90052 & 0.89949 \\
			0.97295 & 0.97192 & 0.97270 & 0.97264 \\
			0.95878 & 0.95612 & 0.95729 & 0.95717 \\
			0.96917 & 0.96766 & 0.96821 & 0.96837 \\ \hline
			0.91747 & 0.91433 & 0.91653 & 0.91528 \\
			0.94419 & 0.93880 & 0.94004 & 0.93967 \\ \hline
		\end{tabular}
		
		\caption{\footnotesize Prediction errors (left) and $R_a^2$ values (right) for the different models under \textbf{Scenario 2a}.}
		\label{Sc2a}
	\end{table}
	\begin{table}[h]
		\centering\footnotesize
		\begin{tabular}{|r|rrrr|}
			\hline
			\diagbox[height=1.1\line]{$p$}{ $n$}	&     100 &     300 &     500 &     700 \\ \hline
			6 & 0.63168 & 0.59355 & 0.59966 & 0.59673 \\
			10 & 0.33911 & 0.32504 & 0.31936 & 0.32095 \\
			14 & 0.28649 & 0.26921 & 0.26497 & 0.26292 \\
			18 & 0.29568 & 0.26385 & 0.25990 & 0.25734 \\ \hline
			$L^2_4$ & 0.45911 & 0.44992 & 0.43995 & 0.45183 \\
			$L^2_6$ & 0.37263 & 0.36384 & 0.36069 & 0.35768 \\ \hline
		\end{tabular}\hspace{.2cm}
		\begin{tabular}{|rrrr|}
			\hline
			100 &     300 &     500 & 700     \\ \hline
			0.99295 & 0.99317 & 0.99319 & 0.99311 \\
			0.99803 & 0.99804 & 0.99805 & 0.99804 \\
			0.99868 & 0.99870 & 0.99871 & 0.99870 \\
			0.99872 & 0.99875 & 0.99876 & 0.99875 \\ \hline
			0.99612 & 0.99606 & 0.99616 & 0.99611 \\
			0.99744 & 0.99752 & 0.99752 & 0.99754 \\ \hline
		\end{tabular}
		\caption{\footnotesize Prediction errors (left) and $R_a^2$ values (right) for the different models under \textbf{Scenario 2b}.}
		\label{Sc2b}
	\end{table}

	\begin{table}[h]
		\centering\footnotesize
		\begin{tabular}{|r|rrrr|}
			\hline
			\diagbox[height=1.1\line]{$p$}{ $n$}	&     100 &     300 &     500 &     700 \\ \hline
			6 & 0.20752 & 0.20636 & 0.20441 & 0.20368 \\
			10 & 0.20912 & 0.20548 & 0.20377 & 0.20279 \\
			14 & 0.21606 & 0.20657 & 0.20454 & 0.20347 \\
			18 & 0.22349 & 0.20865 & 0.20575 & 0.20391 \\ \hline
			$L^2_4$ & 0.20259 & 0.20291 & 0.20177 & 0.20106 \\
			$L^2_6$ & 0.20520 & 0.20406 & 0.20249 & 0.20151 \\ \hline
		\end{tabular}\hspace{.4cm}
		\begin{tabular}{|rrrr|}
			\hline
			100 &     300 &     500 & 700     \\ \hline
			0.91137 & 0.91444 & 0.91282 & 0.91404 \\
			0.91269 & 0.91592 & 0.91422 & 0.91523 \\
			0.91237 & 0.91613 & 0.91436 & 0.91548 \\
			0.91351 & 0.91621 & 0.91455 & 0.91552 \\ \hline
			0.91343 & 0.91624 & 0.91463 & 0.91559 \\
			0.91336 & 0.91617 & 0.91461 & 0.91559 \\ \hline
		\end{tabular}
		\caption{\footnotesize Prediction errors (left) and $R_a^2$ values (right) for the different models under \textbf{Scenario 3}.}
		\label{Sc3}
	\end{table}

	\newpage

	\section{More Simulations}
	To assess the performance of the estimator $\hat{\alpha}_p$,  we have considered several values of $n$ and $p$, for models 2 and 3. The mean  error, over 100 replications ($(1/100)\sum ||\hat{\alpha}_i-\alpha||_{\hat{K}}$) are reported for model 2 in table \ref{tab3} and for model 3 in table \ref{tab4}.  In both tables we assumed that $K$ is known, which is equivalent to known the Hurst parameter. 
	
	\begin{table}[ht]
		\centering
		\small
		\begin{tabular}{|r|rrrr|}
			\hline
			$p$ &       \multicolumn{4}{|c|}{$n$}        \\ \hline
			&   200  &   400  & 600    & 800  \\ \hline
			3 & 0.2610 & 0.2643 & 0.2587 & 0.2789 \\ 
			5 & 0.0075 & 0.0065 & 0.0063 & 0.0062 \\ 
			7 & 0.0643 & 0.0640 & 0.0628 & 0.0609 \\ 
			9 & 0.0642 & 0.0623 & 0.0588 & 0.0607 \\ 
			11 & 0.0511 & 0.0457 & 0.0439 & 0.0427 \\ 
			13 & 0.0266 & 0.0237 & 0.0229 & 0.0224 \\ 
			15 & 0.0057 & 0.0016 & 0.0025 & 0.0019 \\ 
			17 & 0.0222 & 0.0169 & 0.0160 & 0.0157 \\ 
			\hline
		\end{tabular}
		\caption{Mean over 100 replications of $||\hat{\alpha}_p-\alpha||_K^2$. Scenario 2a). $K$ known.}
		\label{tab3}
		
	\end{table}

	\begin{table}[ht]
		\centering\small
		\begin{tabular}{|r|rrrr|}
			\hline
			$p$ &       \multicolumn{4}{c|}{$n$}        \\ \hline
			&    200 &    400 &    600 &    800   \\ \hline 
			3 & 0.9127 & 0.9240 & 1.0097 & 0.9630   \\
			5 & 0.3983 & 0.3774 & 0.3863 & 0.3774   \\
			7 & 0.0695 & 0.0689 & 0.0680 & 0.0678   \\
			9 & 0.0723 & 0.0685 & 0.0684 & 0.0704   \\
			11 & 0.0519 & 0.0444 & 0.0448 & 0.0460   \\
			13 & 0.0210 & 0.0191 & 0.0189 & 0.0186   \\
			15 & 0.0401 & 0.0361 & 0.0350 & 0.0343   \\
			17 & 0.0447 & 0.0401 & 0.0396 & 0.0377   \\ \hline
		\end{tabular}
		\caption{Mean over 100 replications of $||\hat{\alpha}_p-\alpha||_K^2$. Scenario 2b). $K$ known.}
		\label{tab4}
	\end{table}

	\clearpage

	\subsection{The case of unknown $K$}
	
	We assume that the process is a fBm and the value of $H$ is unknown, it was estimated with  $\hat{H}=-\log(\widehat{Var}(X(1/2)))/(2\log(2))$ (which follows from $Var(X(1/2))=(1/2)^{2H}$), in Tables \ref{tab6} and \ref{tab7} we report the mean over 100 replications of $\|\hat{\alpha}_p-\alpha\|_{\hat{K}}$ where 
	$$\hat{K}(s,t)=0.5(|t|^{2\hat{H}}+|s|^{2\hat{H}}-|t-s|^{2\hat{H}})$$

	\begin{table}[ht]
		\centering
		\small
		\begin{tabular}{|r|rrrr|}
			\hline
			$p$ &       \multicolumn{4}{c|}{$n$}        \\ \hline
			&   200  &   400  & 600    & 800  \\ \hline
			3 & 0.3985 & 0.2822 & 0.2341 & 0.2971 \\ 
			5 & 0.0070 & 0.0074 & 0.0075 & 0.0068 \\ 
			7 & 0.0838 & 0.0602 & 0.0539 & 0.0719 \\ 
			9 & 0.0794 & 0.0658 & 0.0584 & 0.0594 \\ 
			11 & 0.0453 & 0.0545 & 0.0397 & 0.0396 \\ 
			13 & 0.0332 & 0.0273 & 0.0224 & 0.0214 \\ 
			15 & 0.0031 & 0.0028 & 0.0013 & 0.0012 \\ 
			17 & 0.0311 & 0.0277 & 0.0136 & 0.0160 \\ 
			\hline
		\end{tabular}
		\caption{Mean over 100 replications of $||\hat{\alpha}_p-\alpha||_K^2$. Scenario 2a). $K$ unknown.}
		\label{tab6}
		
	\end{table}

	\begin{table}[ht]
		\centering\small
		\begin{tabular}{|r|rrrr|}
			\hline
			$p$ &       \multicolumn{4}{|c|}{$n$}        \\ \hline
			&    200 &    400 &    600 &    800   \\ \hline 
			3 & 1.0334 & 0.8965 & 1.0999 & 0.8796 \\ 
			5 & 0.4482 & 0.3608 & 0.4855 & 0.3933 \\ 
			7 & 0.0761 & 0.0718 & 0.0632 & 0.0734 \\ 
			9 & 0.0785 & 0.0427 & 0.0662 & 0.0538 \\ 
			11 & 0.0582 & 0.0667 & 0.0595 & 0.0492 \\ 
			13 & 0.0310 & 0.0223 & 0.0226 & 0.0182 \\ 
			15 & 0.0483 & 0.0488 & 0.0366 & 0.0367 \\ 
			17 & 0.0614 & 0.0396 & 0.0404 & 0.0486 \\ \hline
		\end{tabular}
		\caption{Mean over 100 replications of $||\hat{\alpha}_p-\alpha||_K^2$. Scenario 2b). $K$ unknown.}
		\label{tab7}
	\end{table}

	\subsection{Real data examples}
	
	Real data sets provide another natural playing field for a fair comparison on the prediction capacity of different regression models. 
	In all considered cases the sample is randomly divide in two parts: 80\% of the observations is used for training (i.e. for parameter estimation) and the remaining 20\% is used in order to check the accuracy of the predictions. This random splitting is repeated 100 times. As in the simulated examples, the tables below report the average prediction errors and the (average) adjusted coefficients of determination. 
	
	\
	
	\noindent \textit{The data sets under study}
	\begin{itemize}
		\item[(a)] \textit{The Tecator data set}. This data set has been used and described many times in papers and textbooks; see, e.g., \cite{ferraty2006nonparametric}. It is available in the R-package \tt fda.usc.\rm, see \cite{fda.usc}.  After removing some duplicated data, we have 193 functions obtained from a spectrometric study performed on meat samples in which the near infrared absorbance spectrum is recorded. The response variable is the fat content of the meat pieces. The functions are observed at a grid of 100 points. 
		
		\begin{table}[ht]
			\centering
			\footnotesize
			\begin{tabular}{|l|rrrrrrrr|r|}
				\hline
				$p$ &	4 &      5 &      6 &      7 &      8 &      9 &     10 &     11 &  $L^2_3$\\ \hline
				Pred. error	& 3.8743 & 3.2116 & 2.9689 & 3.3557 & 2.8581 & 3.3886 & 3.4155 & 3.5514 &  3.2567\\
				$R_a^2$ &	0.9106 & 0.9427 & 0.9513 & 0.9436 & 0.9548 & 0.9400 & 0.9446 & 0.9472 & 0.9415\\\hline
				Pred. error &	8.8023 & 3.3377 & 3.3791 & 3.3950 & 3.0646 & 3.3215 & 2.9751 & 2.9421 & 8.4414 \\
				$R_a^2$ &	0.5287 & 0.9396 & 0.9401 & 0.9383 & 0.9527 & 0.9450 & 0.9532 & 0.9588 & 0.5938\\
				\hline
			\end{tabular}
			\caption{\footnotesize Standard prediction errors (mean over 100 replications) and adjusted $R_a^2$, for different values of $p$ for the Tecator dataset. In the first two rows the second derivatives are used. The last two rows correspond to the original data.}
			\label{tecator}
		\end{table}

		An important aspect of this dataset is the fact that the derivatives of the sample functions seem to be more informative than the original data themselves. Thus, we have taken into account this feature, using the original data, as well as the second derivatives to predict the response variable (obtained by preliminary smoothing of the data). The outputs are shown in Table \ref{tecator}, where the first two rows correspond  to the prediction errors and determination coefficients for the second derivatives and the last two rows show the results obtained with the original data. The headings 4,5,...,11 refer to the number $p$ of variables used in the finite-dimensional models and the notation $L^2_3$ refers to the functional $L^2$-model fitted by projecting on 3 principal components, as explained above. The prediction errors are estimated by splitting the sample into training and tests elements, as mentioned in the description of the simulation experiments (all the considered values of $p$ are checked in every run).

		\item[(b)] \textit{The sugar data set}.  This data set has been  previously  considered in functional data analysis by several authors; see e.g. \cite{aneiros2014variable} for additional details. The functions $X(t)$ are fluorescence spectra obtained from sugar samples and the response $Y$ is the ash content, in percentage of the sugar samples.
		The comparison results of finite-dimensional models versus the $L^2$-functional counterpart are shown in Table \ref{sugar}.  Again the outputs correspond to the averages over 100 replications obtained by randomly selecting 100 training and tests samples from the original data.

		\begin{table}[h]
			\centering
			\footnotesize
			\begin{tabular}{|l|rrrrrrrr|c|}
				\hline
				$p$ &		4 &      5 &      6 &      7 &      8 &      9 &     10 &     11 & $L^2_3$ \\ \hline
				Pred. error	&		1.9793 & 2.0472 & 1.9918 & 1.9975 & 1.9300 & 1.9784 & 1.9929 & 1.9938 & 2.2203 \\ 
				$R_a^2$ &		0.7523 & 0.7507 & 0.7607 & 0.7618 & 0.7841 & 0.7733 & 0.7754 & 0.7798 & 0.6579 \\ \hline
			\end{tabular}
			\caption{\footnotesize Standard prediction errors and adjusted $R_a^2$, with different values of $p$ for the sugar dataset.}
			\label{sugar}
		\end{table}

		\item[(c)] \textit{Population data} The data provide, for 237 countries and geographical areas, the percentage of population under 14 years for the period 1968-2018 (one datum per year). In our experiment, the functional data (in fact, they are longitudinal data) are given for vectors $(X(1960), X(1961),\ldots,X(2010))$ and the aim is to predict the value eight years ahead. Thus, the response variable is $Y=X(2018)$. Several theoretical assumptions (for example, independence), commonly used in the linear model, are violated here but, still, our comparisons make sense at an exploratory data level.  The outputs can be found in Table \ref{popu} below. As in the previous examples, they correspond to 100 runs based on random partitions of the data set into 80\% training data and 20\% test data.  Again $p$ denotes the number of years (equispaced in the interval 1960-2010) used as explanatory variables in the finite-dimensional models.  Thus for $p=10$ we consider the years 1960, 1965,...,2010; for $p=8$ we take 1960, 1967, 1974,...,2009. 
		
		\begin{table}[ht]
			\centering
			\footnotesize
			\begin{tabular}{|l|rrrrrrr|r|}
				\hline
				$p$ &		5 &      6 &      7 &      8 &      9 &           11&     13 & $L^2_3$ \\ \hline
				Pred. error	&		1.7298 & 2.7577 & 1.6625 & 1.4712 & 1.5952 &  2.0241 & 1.6138 &   2.1520 \\
				$R_a^2$ &		0.9719 & 0.9294 & 0.9754 & 0.9810 & 0.9778 &  0.9629 & 0.9815 &   0.9559\\\hline
			\end{tabular}
			\caption{\footnotesize Standard error and adjusted $R_a^2$, for different values of $p$ for the population-under-14 dataset. }
			\label{popu}
		\end{table}
	\end{itemize}
	\section{Conclusions}
	We explore a mathematical framework, different from the classical $L^2$-approach \eqref{eq:L2intercept}, for the problem of linear regression with functional explanatory variable $X$ and scalar response $Y$. 
	\begin{itemize}
		\item[(a)] This mathematical formulation includes, as particular cases, the finite-dimensional models \eqref{eq:finite} obtained by considering as explanatory variables a finite set $X(t_1),\ldots,X(t_p)$ of marginals. This would allow for example, to compare such models for variable selection purposes \citep{berrendero2019rkhs} or considering, within a unified framework, the study of asymptotic behaviour of models as the number $p$ of covariates grows to infinity (see e.g. \citep{sur2019modern} for a recent analysis in the logistic regression model). Note also that in the functional case the asymptotic analysis as $p\to\infty$ appears more naturally than in the case of general regression studies,  since the new incoming co-variables are homogeneous in nature as they come from the  unique, predefined reservoir of the one-dimensional marginals of the process $X(t)$.
		
		\item[(b)] From a practical viewpoint, the fact of encompassing all the finite-dimensional models under a unique super-model \eqref{eq.linear-loeve}-\eqref{eq:parzen_model} is also relevant in view of the empirical results of Section \ref{sec:empirical}: indeed, the outputs of the simulations and the real data examples there show that, somewhat surprisingly, there is often little gain (if any) in considering the $L^2$-functional model \eqref{eq:L2intercept} instead of the simpler finite-dimensional alternatives \eqref{eq:finite}.

		\item[(c)] Of course, we don't claim that the $L^2$-based regression model \eqref{eq:L2intercept} should be abandoned in favour of the finite-dimensional alternatives of type \eqref{eq:finite}, since this model is now well-understood and has proven useful in many examples.  We are just suggesting that there are perhaps some reasons to consider the problem of linear functional regression under a broader perspective. In addition, note that the $L^2$ model appears as a particular case of the general formulation \eqref{eq.linear-loeve}-\eqref{eq:parzen_model}. 
		
		\item[(d)] In any case,  even if we are willing to incorporate the finite-dimensional models \eqref{eq:finite}, according to our suggested approach, the functional character of the regression problem is not lost at all as the proposed global general formulation is unequivocally functional.  In practice, this means that, according to our assumptions, our explanatory variables are functions and we cannot get rid of this fact in the formulation of our model. 
		
	\end{itemize}
	\appendix
	\section{Appendix: Proof of Theorem \ref{rkhsnorm}}

	  The proof depends on the following three lemmas, the first one is a direct application of Lemma 3.1 in \cite{bosq1991}, (it is also called  Weyl's inequality in the literature). 
	In what follows, we denote ${\mathcal X}_p$ the $n\times p$ data matrix whose $(i,j)$-entry is $X_i(t_{j,p})$, $i=1,\ldots,n$, $j=1,\ldots,p$. Denote also by $K_{T_p}$, the covariance matrix of $(X(t_{1,p}),\ldots,X(t_{p,p}))$. \color{black} 
	
	\begin{lemma} \label{lemaux2} Let  $\gamma_{1,p}\geq \gamma_{2,p}\geq \dots \geq \gamma_{p,p}$ be the eigenvalues of $K_{T_p}$ and $\hat{\gamma}_{1,p}\geq \hat{\gamma}_{2,p}\geq \dots \geq \hat{\gamma}_{p,p}$ the eigenvalues of $(1/n)(\mathcal{X}_p^\prime \mathcal{X}_p)$. Then, for $j=1,\dots,p$,  $|\gamma_{j,p}-\hat{\gamma}_{j,p}|\leq \|(1/n)(\mathcal{X}_p^\prime \mathcal{X}_p)-K_{T_p}\|_{op}.$
	\end{lemma}

	\begin{lemma}\label{lemaux3} Let   $K$ be a continuous covariance function and  let $T_p$ be a \color{black} set of grid points as in Lemma \ref{parzen}.
		Then   $\lim_{p\to\infty}\frac{1}{p}\|K_{T_p}\|_{op} =\lambda_1$, \color{black} where $\lambda_1$ is the largest  eigenvalue of the covariance operator $\mathcal{K}$ associated with $K$.
	\end{lemma}

\begin{proof}
	It is enough to prove that $\|(1/p)K_{T_p}\|_{op}\to \lambda_1$. Let the $p$-dimensional vector $f_p=(f(t_{1,p}),\dots,f(t_{p,p}))$ be an eigenvector of $(1/p)K_{T_p}$ associated to $\gamma_{1,p}$ the largest eigenvalue of $K_{T_p}$, such that $\|f_p\|_{\max}\leq \gamma_{1,p}$ for all $p$. Let us define a polygonal function $g_p:[0,1]\rightarrow \mathbb{R}$ such that $g_p(t_{i,p})=f(t_{i,p})$, observe that $\|g_p\|_{\infty}=\|f_p\|_{\max}$.  	 Let us prove that $\{g_p\}_p$ is an equicontinuous sequence. Since $K(s,t)$ is continuous, it is also uniformly continuous on $[0,1]^2$. Then, for all $\epsilon>0$, there exists $\delta=\delta(\epsilon)>0$ such that $|K(x,y)-K(x',y')|<\epsilon$ if $\|(x,y)-(x',y')\|_{\max} <\delta$.  Let us denote $\|T_p\|=\max_{i=1,\dots,p-1}|t_{i+1,p}-t_{i,p}|$.  Let $\epsilon>0$ and $p$ large enough such that $\|T_p\|<\delta$. From $\|f_p\|_{\max}\leq \gamma_{1,p}$, it follows that, for all $i,k$ such that $1\leq i \leq p$, $1\leq i+k\leq p$ and $|t_{i,p}-t_{i+k,p}|<\delta$, we have 
	$$
	\gamma_{1,p}|f(t_{i,p})-f(t_{i+k,p})|= \frac{1}{p}\Big|\sum_{j=1}^p \Big[ K(t_{i,p},t_{j,p})-K(t_{i+k,p},t_{j,p})\Big]f(t_{j,p})\Big|\leq \epsilon \gamma_{1,p}.
	$$	
	Then, for $p$ large enough, 	for all $i,k$ such that $1\leq i \leq p$, $1\leq i+k\leq p$ and $|t_{i,p}-t_{i+k,p}|<\delta$, $|f(t_{i,p})-f(t_{i+k,p})|\leq \epsilon$.
	Hence, $\{g_p\}_p$ is equicontinuous. 
	
	Since $\{g_p\}_p$ is bounded, by Arzela-Ascoli there exists $p_k\to\infty $ and $g$ such that $\|g_{p_k}-g\|_\infty\to 0$. For ease of writing we will denote $g_{p_k}=g_p$,
	Let us fix $t_{i,p}$. Then, for all $\epsilon>0$ and for $p$ (which depends on $t_{i,p}$) large enough,  
	\begin{equation}\label{eqauxlem}
		\Big|\gamma_{1,p} g_p(t_{i,p})-\int_0^1 K(t_{i,p},t)g_p(t)dt\Big|= \Big|\frac{1}{p}\sum_{j=1}^p K(t_{i,p},t_{j,p})g_p(t_{j,p})-\int_0^1 K(t_{i,p},t)g_p(t)dt\Big|<\epsilon.
	\end{equation}
	Since $g_p\to g$ uniformly, there exists $p_0>0$ such that Equation \eqref{eqauxlem} holds for all $p>p_0$. By continuity of $K$ and $g_p$ it can be seen that there exists $p_1>p_0$ such that for all $p>p_1$,
	$$\max_{s\in [0,1]}\Big|\gamma_{1,p} g_p(s)-\int_0^1 K(s,t)g_p(t)dt\Big| <\epsilon.$$
	
	Again using that $g_p\to g$ uniformly, $\|\int_0^1 K(s,t)g_p(t)dt-\int_0^1 K(s,t)g(t)dt\|_\infty\to 0$, then $\gamma_{1,p}g_p\to \lambda g$ for some $\lambda$. This proves that  $g$ is an eigenfunction of $\mathcal{K}$ with eigenvalue $\lambda$.
	
	Let us prove that $\lambda= \lambda_1$. Assume first that $\lambda<\lambda_1$ and let $f$ be an eigenfunction associated to $\lambda_1$. Let us define $f_p=(f(t_{1,p}),\dots,f(t_{p,p}))$. Since $f\in L^2[0,1]$ and $f$ is a continuous function, we have that $w_p:=\sum_{i=1}^p f(t_{i,p})^2\to \infty$ as $p\to \infty$, from where it follows that  $\xi_p:=f/\|f_p\|_2^2$    is a Cauchy sequence w.r.t to the uniform norm (observe that $\|\xi_p-\xi_{p+k}\|_\infty\leq \|f\|_\infty((1/w_p)+(1/w_{p+k}))$). Then we can apply \eqref{eqauxlem} for $p>p_2$ for some $p_2$ and we get
	$\|\lambda_1\xi_p-(1/p)K_{T_p}\xi_p\|_{\infty}<\epsilon$. Let us take $\epsilon<(\lambda_1-\lambda)/4$ and $p$ large enough such that $|\gamma_{1,p}-\lambda|<\epsilon/4$. Since $\|\xi_p\|=1$, 
	$$\lambda+\epsilon/4\geq \gamma_{1,p}\geq \|(1/p)K_{T_p}\xi_p\|_{op}\geq \lambda_1-\epsilon,$$
	which contradicts that $\epsilon<(\lambda_1-\lambda)/4$ and then $\lambda\geq \lambda_1$.
	
	To prove the other inequality let $(f(t_{1,p}),\dots,f(t_{p,p}))$ be an eigenvector of $K_{T_p}$ and $g_p$ a polygonal function such that $g_p(t_{i,p})=f(t_{i,p})$, we have proved that there exists $g$ an eigenfunction (associated to $\lambda$) of $\mathcal{K}$, such that $g_{p_k}\to g$ uniformly. Denote $g_p=g_{p_k}$ and define the sequence of functions 
	$$h_p(s)=\frac{1}{\sqrt{\|g\|_2}}g_p(s)\quad s\in [0,1].$$
	Then, $h_p\to h$ uniformly, where $h$ is an eigenfunction of $\mathcal{K}$ associated to $\lambda$, and $\|h\|_2^2=1$. Hence,   $\lambda=\|\mathcal{K}(h)\|_2\leq \lambda_1$. 
	
	Finally, $\gamma_{1,p}\to \lambda_1$ since  we have proved that every subsequence of $\gamma_{1,p}$ has a subsequence which converges to $\lambda_1$.\\
\end{proof}

	\begin{lemma}   \label{asconvlem}Under the hypotheses of Th. \ref{rkhsnorm}. We have, for all $\epsilon>0$,
		\begin{equation} \label{asconv}
			\Big\|\frac{1}{n}\mathcal{X}_p^\prime \mathcal{X}_p-K_{T_p}\Big\|_{op}\leq \epsilon\gamma_{p,p}, \mbox{ eventually, with probability one.}
		\end{equation}
	\end{lemma}\color{black} 
	\begin{proof} Let us define $\mathcal{F}_n= \{\omega:\|\frac{1}{n}\mathcal{X}_p^\prime \mathcal{X}_p-K_{T_p}\|_{op}> \epsilon\gamma_{p,p}\}.$
		To prove (\ref{asconv}), by Borel-Cantelli lemma, it is enough to prove  that, $\sum_n \mathbb{P}(\mathcal{F}_n)<\infty$.
		Let us denote $\mathcal{A}_{i}=\mathcal{A}_{i,n}=\{\omega:\max_{j=1,\dots,p}|X_i(t_{j,p})|<\log^2n\}$, then $\mathbb{P}(\mathcal{F}_n)\leq \mathbb{P}(\mathcal{F}_n\cap \bigcap_{i=1}^n \mathcal{A}_i)+n\mathbb{P}(\mathcal{A}^c_1):=I_1+I_2.$
		
		To prove that $\sum_n I_1<\infty$, we will use Corollary 5.2 in \cite{mackey2014}. Let us define $Y_k$ the $p\times p$ random matrix whose entry $i,j$ is $(X_k(t_{i,p})X_k(t_{j,p})-(K_{T_p})_{i,j})\mathbb{I}_{\mathcal{A}_k}.$
		Let us denote  $Z_k=(X_k(t_{1,p}),\dots,X_k(t_{p,p}))'\mathbb{I}_{\mathcal{A}_k}$, then $Y_k=Z_k Z'_k-K_{T_p}\mathbb{I}_{\mathcal{A}_k}$,  so  $\|Y_k\|_{op}\leq \|Z_k Z'_k \|_{op}+\|K_{T_p}\mathbb{I}_{\mathcal{A}_k}\|_{op}\leq \|Z_k\|_2^2+\|K_{T_p}\mathbb{I}_{\mathcal{A}_k}\|_{op}$. 
		We have that $\|Z_k\|_2^2\leq p\log^2 n$ and
		$$\|K_{T_p}\mathbb{I}_{\mathcal{A}_k}\|_{op} =\max_{z\in S^{p-1}} \mathbb{E}[(z^\prime Z_k)(Z_k^\prime z)]=\max_{z\in S^{p-1}} \mathbb{E}[(z^\prime Z_k)^2]\leq \|z\|_2^2\|Z_k \|_2^2\leq p\log^2 n.$$ 
		To bound $\eta^2:=\|\sum_k \mathbb{E}(Y_k^2)\|_{op}\leq n \|\mathbb{E}(Y^2_1)\|_{op}$, observe that,  
		$\mathbb{E}[Y_1^2]\leq \mathbb{E}[(Z_1 Z'_1)^2]=\mathbb{E}[\|Z_1\|_2^2Z_1 Z'_1]\leq p\log^2 n\mathbb{E}[Z_1 Z'_1]\leq p\log^2 nK_{T_p}$. Then, $\eta^2\leq np\log^2 n\gamma_{1,p}$.   By Corollary 5.2 in \cite{mackey2014}, 	
		 	\begin{equation}
			\mathbb{P}\Big(\mathcal{F}_n \cap \bigcap_{i=1}^n \mathcal{A}_i\Big) \leq p \exp\Bigg[-\frac{(n\epsilon\gamma_{p,p})^2}{3np\gamma_{1,p}\log^2 n+4pn\epsilon\gamma_{p,p}\log^2 n }\Bigg]:=\exp(-a_n).
		\end{equation}
		From Lemma \ref{lemaux3}, $\gamma_{1,p}/p\to \lambda_1$, then $\sum_n I_1<\infty$ follows from the assumption $n(\gamma_{p,p})^2/(p^2\log(n)^3)\to \infty$ since this implies $a_n/\log n\to\infty$.	
		Finally, $\sum_n	n\mathbb{P}(\mathcal{A}^c_1)<\infty$ follows from the assumption $\mathbb{P}(\sup_{t\in [0,1]} X(t)>M)\leq \exp(-CM^2)$ for some constant $C>0$ and for all $M>0$.\color{black}
	\end{proof}
	Now, to prove Theorem \ref{rkhsnorm}, let us take $p=p_n$ and $T_p$ as in Lemma \ref{lemaux3}.   Recall that
	\[
	\alpha_{p}(\cdot) =\sum_{j=1}^{p}\beta_{j,p} K(t_{j,p},\cdot) \mbox{ and } \hat{\alpha}_p(\cdot)=\sum_{j=1}^{p} \hat{\beta}_{j,p}K(t_{j,p},\cdot),
	\]
	and denote $\beta_p=(\beta_{1,p},\ldots,\beta_{p,p})'$, and $\hat\beta_p = (\hat\beta_{1,p},\ldots,\hat \beta_{p,p})'$, where $'$ stands for the transpose.
	Observe that $\|\hat{\alpha}_p-\alpha_{p}\|_K^2 = (\hat{\beta}_p-\beta_{p})'K_{T_p}(\hat{\beta}_p-\beta_{p})=(\hat{\beta}_p-\beta_{p})'K_{T_p}^{1/2} K_{T_p}^{1/2}(\hat{\beta}_p-\beta_{p}).$  Since $K_{T_p}^{1/2}=(K_{T_p}^{1/2})'$, 
	$\|\hat{\alpha}_p-\alpha_{p}\|_K^2 = \|K^{1/2}_{T_p}(\hat{\beta}_p-\beta_{p})\|_2^2\leq \|K_{T_p}^{1/2}\|_{op}^2\|\hat{\beta}_p-\beta_{p}\|_2^2.$ By Lemma \ref{lemaux3}, we can take $p$ large enough such that, $\|\hat{\alpha}_p-\alpha_p\|^2_K\leq 2\lambda_1p\|\hat{\beta}_p-\beta_p\|_2^2.$ So is is enough  to prove 
	\begin{equation}\label{convasbeta}
		\nu_n p\| \hat{\beta}_p-\beta_p\|_2^2\to 0 \quad \mbox{a.s.}
	\end{equation}
	   Since $\hat{\beta}_p=\text{argmin}_\nu  \|Y-\mathcal{X}_p\nu\|_2$,
	we have  $\|Y-\mathcal{X}_p\hat{\beta}_p\|_2\leq \|Y-\mathcal{X}_p\beta_p\|_2=\|\mathbf{e}\|_2$ and  
	$$\|Y-\mathcal{X}_p\hat{\beta}_p\|_2^2=\|\mathcal{X}_p\beta_p+\mathbf{e}-\mathcal{X}_p\hat{\beta}_p\|_2^2=\|\mathcal{X}_p(\hat{\beta}_p-\beta_p)\|_2^2+\|\mathbf{e}\|_2^2-
	2\mathbf{e} '\mathcal{X}_p(\hat{\beta}_p-\beta_p)$$
	so   $\|\mathcal{X}_p(\hat{\beta}_p-\beta_p)\|_2^2\leq 2\mathbf{e}^\prime \mathcal{X}_p(\hat{\beta}_p-\beta_p).$ Denote by $\Phi$ an $n\times p$ matrix whose columns form an orthonormal basis for the   linear space spanned by the columns of  $\mathcal{X}_p$. Then $\mathcal{X}_p(\hat{\beta}_p-\beta_p)=\Phi v$ \color{black} for some unique $v\in \mathbb{R}^p$, and $\tilde{e}=\Phi^\prime \textbf{e}$, then   denoting by $S^{p-1}$ the unit sphere of ${\mathbb R}^p$,\color{black}

	\begin{equation*}\label{eq:ez}
		\|\mathcal{X}_p(\hat{\beta}_p-\beta_p)\|_2\leq 2\textbf{e}^\prime\frac{\mathcal{X}_p(\hat{\beta}_p-\beta_p)}{\|\mathcal{X}_p(\hat{\beta}_p-\beta_p)\|_2}=2\tilde{e}^\prime  \frac{v}{\|v\|_2}\leq   2\sup_{u\in S^{p-1}}\tilde{e}^\prime u.\color{black}
	\end{equation*}
	
	\color{black}
	Let us denote $N_{1/2}$ a minimal covering of $S^{p-1}$ with balls of radius $1/2$, centred at points in $S^{p-1}$, its cardinality $|N_{1/2}|$ is bounded from above by $5^{p-1}$. For all   $u\in S^{p-1}$ there exist a point $z$ in the set $C_{1/2}$ of centres of the balls in $N_{1/2}$ and $w\in{\mathbb R}^p$ such that $u=z+w$, with $\|w\|_2\leq 1/2$. Denote $W_{1/2}$ the set of such $w's$ so   $\max_{u \in S^{p-1}} \tilde{e}^\prime  u\leq \max_{z\in C_{1/2}} \tilde{e}^\prime z \ \ + \ \ \max_{w\in W_{1/2}} \tilde{e}^\prime w,$ then 
	\begin{equation}\label{eq:zmax}
		2\sup_{u\in S^{p-1}}\tilde{e}^\prime u\leq 4\max_{z\in C_{1/2}} \tilde{e}^\prime z.
	\end{equation}\color{black}
	Observe that $\|K_{T_p}^{1/2}(\hat{\beta}_p-\beta_p)\|_2^2=(\hat{\beta}_p-\beta_p)^\prime K_{T_p}(\hat{\beta}_p-\beta_p)$.  Thus,
	\small
	\begin{align*}\label{lastbound}
		\|K_{T_p}^{1/2}(\hat{\beta}_p-\beta_p)\|_2^2=(\hat{\beta}_p-\beta_p)^\prime (K_{T_p}-(1/n)(\mathcal{X}_p^\prime \mathcal{X}_p))(\hat{\beta}_p-\beta_p)+
		(\hat{\beta}_p-\beta_p)^\prime (1/n)(\mathcal{X}_p^\prime \mathcal{X}_p)(\hat{\beta}_p-\beta_p).
	\end{align*}\color{black}
	\normalsize 
	  Now, using Lemma \ref{asconvlem} together with the inequalities $|x^\prime A x|\leq \|A\|_{op}\|x\|^2$, for $x=(\hat \beta_p-\beta_p)$, $A=(1/n)(\mathcal{X}_p^\prime \mathcal{X}_p)-K_{T_p}$ and  $\|K_{T_p}\|_{op}\geq \gamma_{p,p}$ we get that, eventually a.s., 
	$$\frac{1}{n}\|\mathcal{X}_p(\hat{\beta}_p-\beta_p)\|_2^2\geq \|K_{T_p}^{1/2}(\hat{\beta}_p-\beta_p)\|_2^2-\epsilon \gamma_{p,p}\|\hat{\beta}_p-\beta_p\|_2^2\geq \|\hat{\beta}_p-\beta_p\|_2^2\gamma_{p,p}-\epsilon \gamma_{p,p}\|\hat{\beta}_p- \beta_p\color{black}\|_2^2.$$
	Then, from \eqref{eq:zmax}, $
	\mathbb{P}(\nu_n p\| \hat{\beta}_p-\beta_p\|_2^2> \epsilon)\leq  5^{p-1}\max_{z\in S^{p-1}} \mathbb{P}\big(\tilde{e}^\prime z>\sqrt{n\epsilon \gamma_{p,p}(1-\epsilon)/(2p \nu_n)}\big).$ Now, note that there is a sub-Gaussianity bound, not depending on $z$, for the tail probabilities $\mathbb{P}(\tilde{e}^\prime z>t)$  (see the remarks immediately before Theorem \ref{rkhsnorm}). Finally, \eqref{convasbeta} follows from  $n\gamma_{p,p}/(p^2\nu_n\log n)\to \infty$. 
	
	
	\bibliographystyle{apalike}
	\bibliography{Refs}

\end{document}